\documentclass[12pt,reqno]{amsart}
\usepackage{amsmath}
\usepackage{amssymb}
\usepackage{amscd}

\newtheorem{theorem}{Theorem}

\newtheorem{theoremb}{Theorem}
\newtheorem{theoremc}{Theorem}

\newtheorem{rk}[theoremc]{Remark}
\newtheorem{cor}[theoremb]{Corollary}

\newtheorem{prop}[theorem]{Proposition}
\newenvironment{Proof}[1]{\textbf{#1.} }

\newcommand\bib[1]{\bibitem[#1]{#1}}

\renewcommand\a{\alpha}
\renewcommand\b{\beta}

\newcommand\C{{\mathcal C}}
\newcommand\com[1]{}

\renewcommand\d{\delta}
\newcommand\D{{\mathcal D}}
\newcommand\e{\eta}
\newcommand\E{\mathcal{E}}

\newcommand\g{{\frak g}}

\newcommand\h{{\frak h}}

\newcommand\La{\Lambda}
\newcommand\m{\frak m}

\newcommand\oo{\omega}
\newcommand\op[1]{\mathop{\rm #1}\nolimits}
\newcommand\ot{\otimes}
\newcommand\p{\partial}

\newcommand\R{{\mathbb R}}
\renewcommand\t{\tau}
\newcommand\sym{\op{sym}}

\newcommand\ti{\tilde}

\newcommand\we{\wedge}
\newcommand\x{\xi}
\newcommand\z{\sigma}
\newcommand\Z{{\mathbb Z}}

\begin{document}

 \title[Rank 2 distributions of Monge equations]{Rank 2 distributions of Monge equations:\\
 {\em symmetries, equivalences, extensions}}
 \author{Ian Anderson, Boris Kruglikov}


 \begin{abstract}
By developing the Tanaka theory for rank 2 distributions, we completely classify classical Monge
equations having maximal finite-dimensional symmetry algebras with fixed (albeit arbitrary)
pair of its orders. Investigation of the corresponding Tanaka algebras leads to a new
Lie-B\"acklund theorem. We prove that all flat Monge equations are successive
integrable extensions of the Hilbert-Cartan equation. Many new examples are provided.
 \end{abstract}

 \maketitle

\section*{Introduction and main results}

In his paper of 1895 \cite{L}, S.\,Lie made the important observation that
the simplest differential equations, beyond those of finite type, are equations with general solution
depending upon 1 function of 1 variable, and he indicates that they are integrable by ODE methods.

Such differential equations can generally be encoded as rank 2 distributions or equivalently as
codimension 2 Pfaffian systems on $k$-dimensional manifolds. For $k=5$, Goursat showed \cite{G} that
these distributions can be realized as Monge equations
 \begin{equation}\label{E0}
y'=F(x,y,z,z',z''),
 \end{equation}
while E.\,Cartan \cite{C1} realized them as reductions of involutive pairs of 2nd order
PDEs in 1 dependent and 2 independent variables.
The interplay between the Monge and PDE realizations of rank 2 distributions is an important aspect
of their geometry.

From the Goursat-Monge viewpoint, one can therefore interpret Cartan's 1910 paper as a detailed
analysis of the geometry of Monge equations (\ref{E0}). In what is still considered to be a
computational \textsl{tour de force}, Cartan solved the equivalence problem for (\ref{E0}) and,
as a corollary, deduced that the celebrated\footnote{We write this in retrospection \cite{H,C2};
in \cite{C1} only the PDE model was explicitly constructed.}
Hilbert-Cartan equation
 \begin{equation}\label{HC}
y'=(z'')^2
 \end{equation}
is the Monge equation of type (\ref{E0}) with the largest finite-dimensional Lie algebra of
symmetries.

Despite Lie's observation and Goursat's contribution, or perhaps because of the complexity of
Cartan's analysis, little progress has been made in the symmetry analysis of the
general Monge equations
 \begin{equation}\label{E1}
y^{(m)}=F\bigl(x,y,y',\dots,y^{(m-1)},z,z',\dots,z^{(n)}\bigr)\
 \end{equation}
(here and beyond $y=y(x),z=z(x)$ and we assume $m\le n$).

These underdetermined ordinary differential equations are important in the variational calculus
and control theory, in reductions of compatible PDEs and in the symmetry approach to Darboux integrability.
Their geometry is the subject of the present article.

The first goal of the paper is to describe the most symmetric equations of type (\ref{E1}).
There are different kinds of symmetries and Lie-B\"acklund type theorems clarify the relations
between point, contact, generalized (higher) and internal symmetries whenever possible,
see \cite{C1,KLV,AKO,GK}.

Below we restrict to purely internal symmetries.
We will assume $m>0$, otherwise (\ref{E1}) has an infinite-dimensional
algebra of symmetries.

 \begin{theorem}\label{Thm1}
Among equations $\E$ of type (\ref{E1}) with fixed $m\le n$ and
$F_{z^{(n)}z^{(n)}}\ne0$, the most symmetric (realizing $\max\op{dim}\sym(\E)$)
are internally  equivalent to the equation
 \begin{equation}\label{E2}
y^{(m)}=\bigl(z^{(n)}\bigr)^2.
 \end{equation}
The dimension of the symmetry algebra $\sym(\E)$ for this equation
equals $2n+5$ for $m=1$, $n>2$ and $2n+4$ for $m>1$.
 \end{theorem}
Let us denote by $\E_{m,n}$ the model equation (\ref{E2}).
Notice that the equation $\E_{1,1}$ is internally equivalent to the pure jets $J^2(\R,\R)$
with (infinite-dimensional) contact symmetry algebra, while $\E_{1,2}$
is the Hilbert-Cartan equation (\ref{HC}) with symmetry algebra
$\mathfrak{g}_2=\op{Lie}(G_2)$ of dimension 14. These are the exceptional cases omitted in
Theorem \ref{Thm1}.

 \begin{rk}\label{Rmk1}
As a by-product of the proof of Theorem \ref{Thm1} we conclude that Monge equations $\E$
with $F_{z^{(n)}z^{(n)}}\ne0$ are strongly regular in the sense of Tanaka \cite{T}.
This implies the existence of a Cartan connection for (\ref{E1}) as a $\mathfrak{t}_{m,n}$-valued 1-form
on a natural bundle over the equation $\E$, where $\mathfrak{t}_{m,n}$ is the Tanaka algebra calculated in Section \ref{S32}.
 \end{rk}

As is well-known, equation (\ref{E1}) defines a submanifold $\E$ in the space of jets
$J^{m,n}(\R^1,\R^2)\simeq\R^{m+n+3}(x,y,y_1,\dots,y_m,z,z_1,\dots,z_n)$,
with $\dim\E=m+n+2$. The canonical rank 3 distribution
 $$
\C=\op{Ann}\{dy_i-y_{i+1}\,dx,dz_j-z_{j+1}\,dx\,|\,1\le i<m,1\le j<n\}
 $$
(Cartan distribution) on the jet space $J^{m,n}(\R^1,\R^2)$ induces a rank 2 distribution $\Delta$ on
$\E$. Internal symmetries of (\ref{E1}) are, by definition, symmetries of the pair $(\E,\Delta)$.

As noticed above, rank 2 distributions were investigated in many classical papers. Recently
Doubrov and Zelenko \cite{DZ} described their internal geometry via a canonical frame on a bundle over
the manifold $\E$, providing a unique solution for maximally symmetric model among all rank 2
distributions from some generic stratum (called maximal class in loc.cit). This class does
not appear to have a clear description for distributions, arising for differential equations (\ref{E1}).

Theorem \ref{Thm1} yields maximally symmetric distributions in a large class of
distributions corresponding to Monge equations of a fixed dimension.

 \begin{cor}\label{Cor1}
Among all rank 2 distributions $\Delta$ corresponding to equations (\ref{E1}) with fixed $m+n=k+1$
and $F_{z^{(n)}z^{(n)}}\ne0$, the most symmetric are locally equivalent to the canonical distribution
of the equation $\E_{1,k}$.
 \end{cor}

The non-degeneracy condition $F_{z^{(n)}z^{(n)}}\ne0$
is related to the finite-dimensionality of $\sym(\E)$
and will be explained in Section \ref{S2}. It implies that the growth vector of the
corresponding distribution starts $(2,1,2,\dots)$ and that the distribution is not de-prolongable.

Let us note that, according to \cite{AKO}, internal symmetries for equations of type (\ref{E1}),
in the particular case $m=n=2$, coincide with first order generalized symmetries.
We can prove much more than this, namely, we prove a new Lie-B\"acklund type theorem.

 \begin{theorem}\label{Thm2}
For equations of type (\ref{E1}), with $F_{z^{(n)}z^{(n)}}\ne0$, $0<m\le n$ and $(m,n)$ being different
from the exceptional cases $(1,1)$ and $(1,2)$, the algebra of internal symmetries coincides with
the algebra of external symmetries, induced by Lie transformations of the ambient mixed jet space
$J^{m,n}(\R,\R^2)$.
 \end{theorem}
Lie transformations of $J^{m,n}(\R,\R^2)$ will be described in Section \ref{S4}.
Let us remark that Theorem \ref{Thm2}
is the first instance of Lie-B\"acklund type theorem whose proof entirely depends on the structure of
the graded Lie algebra associated to the distribution, and cannot be derived by the usual de-prolongation arguments.

To prove our results we exploit and elaborate upon the Tanaka theory \cite{T}.
Distributions corresponding to equation (\ref{E2}) are flat in its sense.
There exist however underdetermined ODE systems, for which the respective rank 2 distributions
cannot be realized as non-degenerate Monge equations. Their flat models can nevertheless be calculated
via the same technique.

In Section \ref{S5} we relate the algebraic theory of central extensions of Lie algebras to
the geometric theory of integrable extensions.
We provide a very simple and constructive description of all flat models.

 \begin{theorem}\label{Thm3}
Any flat rank 2 distribution with growth vector $(2,1,2,\dots)$
(be it represented as a nondegenerate single Monge equation or underdetermined system of ODEs)
is a result of successive integrable extensions of the Hilbert-Cartan equation (\ref{HC}).
 \end{theorem}

This allows us to calculate many examples of flat rank 2 distributions and the
corresponding ODE systems.

At the end of the paper we show that our theory does not restrict to ODEs only. We
demonstrate relations of our results to integration of determined and overdetermined PDEs.
In particular, we construct a flat non-trivial integrable extension of Cartan's
most symmetric involutive PDE model of 1910.

The results of this paper indicate that the method of E.\,Cartan from his seminal
paper \cite{C1} can be extended to a much broader class of PDE systems.
This generalization and further applications to the method of Darboux \cite{AFV}
and compatible overdetermined systems \cite{KL}
will be presented in a separate publication.

\medskip

\textsc{Acknowledgment.}
We are grateful to Valentin Lychagin for discussions during the initial stage
of the project. Hospitality of Institut Mittag-Leffler (Stockholm) in 2007, MFO (Oberwolfach)
in 2008 and Banach center (Warsaw) in 2009 is acknowledged.
I.A.\ was supported by NSF grant DMS-0713830.
The computations in Sections \ref{S5} and \ref{S6}
were done, in part, using the \textsf{DifferentialGeometry} package of \textsc{Maple}.

\section{Tanaka algebra of a distribution}\label{S1}

Given a distribution $\Delta\subset TM$, its \emph{weak derived flag}
$\{\Delta_i\}$ is given via the module of its sections by
$\Gamma(\Delta_{i+1})=[\Gamma(\Delta),\Gamma(\Delta_i)]$ with $\Delta_1=\Delta$.
The \emph{strong derived flag} is given by the filtration $\nabla_{i+1}=\p\nabla_i$
with $\Gamma(\nabla_{i+1})=[\Gamma(\nabla_i),\Gamma(\nabla_i)]$ and $\nabla_1=\Delta$.

We will assume throughout this paper that the distribution is \emph{completely non-holonomic},
i.e. $\Delta_k=TM$ for some $k$, and we also assume that $\Delta$ is \emph{regular}, so that the ranks
of both $\Delta_i$ and $\nabla_i$ are constant.

The quotient sheaf $\m=\oplus_{i<0}g_i$, $g_i=\Delta_{-i}/\Delta_{-i-1}$ has a natural structure of the graded
nilpotent Lie algebra (GNLA) at any point $x\in M$. The bracket on $\m$ is induced by the commutator
of vector fields on $M$. $\Delta$ is called \emph{strongly regular} if the GNLA $\m(x)$ does not depend
on $x\in M$ (but we do not impose this property a priori).

The \emph{growth vector} of $\Delta=g_{-1}$ is the sequence of dimensions\linebreak
 $(\dim g_{-1},\dim g_{-2}\dots)$.
(In other sources, the growth vector is one of the sequences $(\dim\Delta_1,\dim\Delta_2,\dots)$
or $(\dim\nabla_1,\dim\nabla_2,\dots)$ but, as our main tool is the Tanaka theory, we choose this
convention).

The Tanaka prolongation $\g=\hat\m$ is the GNLA with negative graded part $\m$, whose non-negative
part is defined successively by the rule \cite{T,Y}: $g_k$ consists of all
$u\in \bigoplus\limits_{i<0}g_{k+i}\ot g_i^*$, acting as derivations:
$u([X,Y])=[u(X),Y]+[X,u(Y)]$, $X,Y\in\m$. Notice that since $\m$ is fundamental, i.e.
$g_{-1}$ generates the whole GNLA $\m$, the grading $k$ homomorphism $u$
is uniquely determined by the restriction $u:g_{-1}\to g_{k-1}$.

The space $\g=\oplus\,g_i$ is naturally a graded Lie algebra, which we will call
the {\em Tanaka algebra\/} of $\Delta$ (it depends on the point $x\in M$, unless $\Delta$ is strongly regular).
By contrast, the GNLA $\m=\oplus_{i<0}g_i$ will be called the {\em Carnot algebra\/} of $\Delta$
(it is also called Levi-Tanaka or symbol algebra in the literature).

We wish to indicate that $\g=\hat\m$ can alternatively be defined via right extensions for modules, just as
right extensions\footnote{Both fundamental extension constructions in the theory of Lie algebras
play a role in this paper: right extensions are important for Tanaka prolongation, while central extensions
appear in Section \ref{S5} in connection with integral extensions.}
for Lie algebras are governed by $H^1(\m,\m)$ \cite{SL,F}.
First notice that the cohomology $H^i$ of GNLA is graded.
We denote the grading by a subscript $H^i=\oplus H^i_k$.

Then $g_0=H^1_0(\m,\m)$ and we define $\g^0=\m\oplus g_0$. On the next step we get
$g_1=H^1_1(\m,\g^0)$ and we define $\g^1=\g^0\oplus g_1$. This is
equivalent to adding variables $w_i$ to kill the cohomology classes $[\oo_i]\in H^1_1(\m,\g^0)$,
that is, let $d w_i=\oo_i$. The $\m$-action naturally extends to $\g^1$ making it
into a representation and we continue: $\g^p=g_{-k}\oplus\dots\oplus g_p$ with
$g_p=H^1_p(\m,\g^{p-1})$. The prolongation of $\m$ is $\g=\lim\g^p$.


The importance of the Tanaka algebra is in majorizing the symmetry algebra $\sym(\Delta)$ of
the distribution. The following is the Corollary of Theorem 8.4 from \cite{T}:
 \begin{theorem}\label{ThT}
Let $(M,\Delta)$ be a manifold equipped with distribution, $\m$ its
Carnot GNLA and $\g=\hat\m$ the Tanaka algebra. Then
 $$
\dim\sym(\Delta)\le\dim\g.
 $$
Moreover, provided $\g$ is finite-dimensional, this inequality is an equality
if and only if the distribution $\Delta$ is locally flat.
 \end{theorem}

Some remarks concerning this theorem are of order. First, this theorem is
formulated (and proved in \cite{T}) for strongly regular systems,
and this will suffice for our purpose.
But it remains true if we understand the right hand size as $\sup_x\dim\g(x)$.

Second, the last part of the theorem is not formulated in \cite{T} as we have, but it follows
from the construction of an absolute parallelism on the prolongation manifold of the structure.

Finally, locally flat means that the structure functions of the absolute parallelism vanish. This is
equivalent to the claim that the distribution is locally diffeomorphic
to the standard distribution on the affine realization of Lie group corresponding to the GNLA $\g$, see \cite{T}.

\section{Finite-dimensionality of the symmetry algebra}\label{S2}

In this section we give a sufficient condition for $\dim\g$, and hence $\dim\sym(\Delta)$,
to be finite in the case $\dim\Delta=2$. The statement has been surprisingly unknown until now.

Notice that for a rank 2 non-holonomic distribution the growth vector always begins like $(2,1,\dots)$
and that the first number after a sequence of 1\!'s can be only 2.

 \begin{theorem}\label{thm1}
Let $\Delta$ be a completely non-holonomic rank 2 distribution with growth vector $(2,1,\dots,1,2,\dots)$.
Then the algebra $\sym(\Delta)$ is finite-dimensional.
 \end{theorem}

 \begin{proof}
Let $\g=\oplus g_i$ be the Tanaka algebra of $\Delta$ evaluated at any point $x\in M$.
In view of Theorem \ref{ThT} it is enough to verify that $\dim\g<\infty$.

To this end we use Corollary 2 of Theorem 11.1 from \cite{T}, which states that if the subalgebra
 $$
\h_0=\{v\in g_0:[v,g_r]=0\ \forall r<-1\}
 $$
is zero, then $\g$ is finite-dimensional.\footnote{The meaning of the condition
$\h_0=0$ is this: the complex characteristic variety of the Lie equation for
$\op{sym}(\Delta)$ is empty.}

Let $-k$ be the grading of the second two-dimensional space $g_{-k}$ in $\m$ ($k>3$).
Choose a basis $e_1,e_1'$ of $g_{-1}$ that generates basis elements
$e_2=[e_1,e_1']$ of $g_{-2}$, $e_3=[e_1,e_2]$ of $g_{-3}$, \dots, $e_{k-1}=[e_1,e_{k-2}]$
of $g_{1-k}$ and $e_k=[e_1,e_{k-1}]$, $e_k'=[e_1',e_{k-1}]$ of $g_{-k}$.

Let $v\in\h_0$ and write $[v,e_1]=a\,e_1+b\,e_1'$, $[v,e_1']=c\,e_1+d\,e_1'$.
Then $[v,e_{k-1}]=0$ implies $[v,e_k]=[a\,e_1+b\,e_1',e_{k-1}]=a\,e_k+b\,e_k'=0$
and $[v,e_k']=[c\,e_1+d\,e_1',e_{k-1}]=c\,e_k+d\,e_k'=0$. Whence $a=b=c=d=0$
and therefore $v=0$. The claim follows.
 \end{proof}

Consider the growth vector of $\Delta$. If it starts $(2,1,2,\dots)$ then,
by Theorem \ref{thm1}, the symmetry algebra is finite-dimensional.
The other possibility, namely $(2,1,1,\dots)$, has a clear geometric meaning
which we provide in the next theorem.

Recall that the (geometric) prolongation of a rank 2 distribution $\bar\Delta$ on $\bar M$ is
(we consider only the local picture) the rank 2 distribution
$\Delta=\langle \cos t\cdot U+\sin t\cdot V,\p_t\rangle$ on $M=\bar M\times S^1$, where $U,V$ are
generators of $\bar\Delta$ and $t$ the coordinate on $S^1$.

If the distribution $\bar\Delta$ has preferred (vertical) section
$V\in\Gamma(\bar\Delta)$, then the (affine) prolongation $\Delta=\bar\Delta^{(1)}$
is simply $\Delta=\langle U+t\,V,\p_t\rangle$ on $M=\bar M\times\R^1$, where $t$ is the coordinate on $\R^1$.
This coincides with (Spencer) prolongation in the geometric theory of PDEs (\cite{KLV,MZ}).

The reverse operation is not generally possible but, in the case when $\Delta=\bar\Delta^{(1)}$,
we shall call the operation $\Delta\mapsto\bar\Delta$ de-prolongation.

 \begin{theorem}\label{ThDeProl}
Let $\Delta$ be a rank 2 non-holonomic distribution with the derived rank 3 distribution
$\Delta_2=[\Delta,\Delta]$. The growth vector of $\Delta$ is $(2,1,1,\dots)$
if and only if there exists a vector field $\xi\in\Gamma(\Delta)$ which is a
Cauchy characteristic for the distribution $\Delta_2$.
In this, and only in this case, $\Delta$ can be locally de-prolonged.
 \end{theorem}

Recall that the Cauchy characteristic $\xi$ is a symmetry of the distribution $\Delta$ and its derived flag.
So de-prolongation in this case is simply the quotient of the first derived distribution by the Cauchy characteristic
 $$
(\bar M,\bar\Delta)=(M,\Delta_2)/\x.
 $$

The above theorem is due to E.\,Cartan \cite{C2}. We present a new proof here for the convenience
of the reader.

\begin{proof}
The claim about the existence of a Cauchy characteristic is obvious.
We shall only demonstrate that $\bar\Delta^{(1)}=\Delta$.

For this, let us write $M$ locally as the product $\R^{n-1}_x\times\R^1_t$
in such a way that $\x=\p_t$. We take the other generator $X$ of $\Delta$ from $T\R^{n-1}_x$
with coefficient dependent on $(x,t)$. The derived flag equals
 $$
\Delta_1=\Delta=\langle X,\p_t\rangle,\ \Delta_2=\langle X,\dot X,\p_t\rangle,\
\Delta_3=\langle X,\dot X,\ddot X,[X,\dot X],\p_t\rangle,
 $$
where dot denotes the differentiation by $t$ (or equivalently, the Lie derivative along $\p_t$).
Since the growth vector is $(2,1,1,\dots)$ and $\p_t$ is the Cauchy characteristic, we must have
 $$
\ddot X=\a\,X+\b\,\dot X
 $$
for some $t$-dependent functions $\a,\b\in C^\infty(\bar M\times\R)$. Let $W=u\,X+v\,\dot X$
with $u,v\in C^\infty(\bar M\times\R)$.
Then solution to the system
 $$
\dot W=(\dot u+\a\,v)X+(u+\dot v+\b\,v)\dot X=0
 $$
is equivalent (provided $\op{rank}(X,\dot X)=2$) to the ODE system
 $$
\frac{d}{dt}\begin{pmatrix}u \\ v\end{pmatrix}=\begin{pmatrix}0 & -\a \\ -1 & -\b\end{pmatrix}
\begin{pmatrix}u \\ v\end{pmatrix}.
 $$
By choosing a basis of fundamental solutions, we get two $t$-independent vector fields $U,V$ such
that $\langle U,V\rangle=\langle X,\dot X\rangle$. Since generators are defined
up to scaling, we can assume $X=U+b\,V$. Regularity implies $\dot b\ne0$. By reparametrizing $t$ we can
assume that $b=t$. Thus we have:
 $$
\Delta=\langle U+tV,\p_t\rangle,\ \Delta_2=\langle U,V,\p_t\rangle,\
\bar\Delta=\langle U,V\rangle,
 $$
so that the claim follows from the definition of the prolongation.
\end{proof}

De-prolongation does not just result in the removal of 1 from the growth vector at the first component
and shift. It can significantly change the weak derived flag.

{\bf Example.}
Consider the Hilbert-Cartan equation $\E_{1,2}$.
It has growth vector $(2,1,2)$ and the Tanaka algebra is the exceptional Lie algebra $\mathfrak{g}_2$
with grading (the number in the sequence below means the dimension while the subscript - the grading):
 $$
(2_{(-3)},1_{(-2)},2_{(-1)},4_{(0)},2_{(1)},1_{(2)},2_{(3)}).
 $$
The prolongation of $\E_{1,2}$ is a 6-dimensional manifold with rank 2 distribution of
growth $(2,1,1,1,1)$. Here we see that de-prolongation is not just removal of 1 and shift.

The Tanaka algebra of this 6-dimensional model is still $\mathfrak{g}_2$, but now with the grading
 $$
(1_{(-5)},1_{(-4)},1_{(-3)},1_{(-2)},2_{(-1)},2_{(0)},2_{(1)},1_{(2)},1_{(3)},1_{(4)},1_{(5)}).
 $$
The next prolongation has growth vector $(2,1,1,1,1,1)$, and at this point the Tanaka prolongation
is the contact (infinite-dimensional) algebra. This means that the 2nd (Spencer) prolongation of $\E_{1,2}$
is not Tanaka-flat.

\medskip

Since the algebras $\sym(\bar\Delta)$ and $\sym(\bar\Delta^{(1)})$ are isomorphic,
Theorems \ref{ThDeProl} and \ref{thm1} imply the following important statement
(it can be also deduced in a non-trivial way from Theorem 7.1 of \cite{MT}).

 \begin{cor}
A regular germ of a rank 2 distribution $\Delta$ has infinite-dimensional symmetry algebra only in one of the following cases.
 \begin{enumerate}
\item $\Delta$ completely non-holonomic: then it is equivalent to the
  Cartan distribution $\mathcal{C}_k$ on jet-space $J^k(\R,\R)$, with $k=\dim M-2$.
\item $\Delta_i\ne TM$ for any $i$: then there exists a codimension 1 foliation $\mathcal{F}$ of $M$
  with $T\mathcal{F}\supset\Delta$ such that all leaves $(\mathcal{F},\Delta)$ are equivalent.
 \end{enumerate}
 \end{cor}

Recall that distributions $\Delta$ with strong growth vector $(2,1,1,\dots,1)$ are called
Goursat distributions \cite{G}. Their regular points can be characterized by the condition that
the growth vector is $(2,1,1,\dots,1)$ in both the weak and the strong sense.
Near such points the normal form
 $$
\mathcal{C}_k=\op{Ann}\{dy_i-y_{i+1}\,dx|0\le i<k\}\subset TJ^k(\R^1,\R^1)
 $$
is provided by the von Weber - Cartan theorem \cite{W,C2}, see also \cite{GKR}.

\smallskip

This concludes our discussion of distributions with growth vector $(2,1,1,\dots)$.
From this point forward we assume that the distribution is not de-prolongable,
i.e. its growth is $(2,1,2,\dots)$.

\section{Most symmetric rank 2 distributions}

In this section we study the Tanaka algebra of a single Monge ODE and prove our main result.

\subsection{Calculation of the Carnot algebra.}\label{S31}
Let us consider equation (\ref{E1}) with fixed $m\le n$ as a submanifold in the space of jets
 $$
\E\subset J^{m,n}(\R,\R^2)\simeq\R^{n+m+3}(x,y,\dots,y_m,z,z_1,\dots,z_n),
 $$
given as a hypersurface $y_m=F(x,y,\dots,y_{m-1},z,z_1,\dots,z_n)$.
Internally $\E\simeq\R^{n+m+2}(x,y,\dots,y_{m-1},z,z_1,\dots,z_n)$ is
canonically equipped with the rank 2 distribution
 $$
\C=\langle\D_x=\p_x+y_1\p_y+\dots+F\p_{y_{m-1}}+z_1\p_z+\dots+z_n\p_{z_{n-1}},\p_{z_n}\rangle.
 $$
Its (weak) derived flag has growth vector $(2,1,2,\dots)$ if and only if
 $$
F_{z_nz_n}\ne0.
 $$
In what follows we will assume this \emph{non-degeneracy condition} which,
according to Theorem \ref{thm1}, implies $\dim\sym(\C)<\infty$.

In this case the derived distribution is $\C_2=\C+\langle\p_{z_{n-1}}+F_{z_n}\p_{y_{m-1}}\rangle$
and the second derived is
 $$
\C_3=\langle\D_x,\p_{y_{m-1}},\p_{z_n},\p_{z_{n-1}},\p_{z_{n-2}}+F_{z_n}\p_{y_{m-2}}\rangle.
 $$
The next (weak) derived\footnote{These formulas imply that the weak and strong flags of $\C$ coincide,
as was mentioned in Remark \ref{Rmk1}; this plays a role in Section~\ref{S4}.\label{ftnFLAG}}
are:
 \begin{multline*}
\C_{i+1}=\langle\D_x,\p_{y_{m-1}},\dots,\p_{y_{m-i+1}},\p_{z_n},\dots\p_{z_{n-i+1}},\p_{z_{n-i}}+F_{z_n}\p_{y_{m-i}}\rangle,\\
\C_{j+1}=\langle\D_x,\p_{y_{m-1}},\dots,\p_y,\p_{z_n},\dots\p_{z_{n-j}}\rangle,\qquad i\le m<j.
 \end{multline*}
Thus the dimension grows by two from grading $-3$ till grading $-m-2$ (resp. $-m-1$ if $n=m$) and then grows by one.
The growth vector is:
 $$
(2,1,\underbrace{2,\dots,2}_m,\underbrace{1,\dots,1}_{n-m-1})\quad\text{ for }m<n
 $$
and
 $$
(2,1,\underbrace{2,\dots,2}_{n-1},1)\quad\text{ for }m=n.
 $$

 \begin{prop}\label{L1}
Carnot algebra of equation (\ref{E1}) does not depend on a point $x\in\E$ and
is equal to the Lie algebra $\mathfrak{f}_{m,n}$ given by the following table of brackets
(only upper-diagonal non-zero entries are shown). For $m<n-1$:
 \begin{center}
 \begin{tabular}{c|ccccccccccccccc}
\!\!& $e_1$ & $e_1'$ & $e_2$ & $e_3$ & $e_3'$ & $e_4$ & $e_4'$ & \!\!\dots\!\!
 & $e_{m+2}$ & $e_{m+2}'$ & $e_{m+3}$ & \!\!$e_{m+4}$ & \!\!\dots\!\! & $e_{n+1}$\!\!\!\\
 \hline
\!\!$e_1$ & & $e_2$ & $e_3$ & $e_4$ & $e_4'$ & $e_5$ & $e_5'$ & \!\!\dots\!\!
 & $e_{m+3}$ & & $e_{m+4}$ & \!\!\dots\!\! & \!\!\!$e_{n+1}$\!\!\! \\
\!\!$e_1'$ & & & $e_3'$ & $e_4'$ & & $e_5'$ &  & \!\!\dots\!\! \\
\!\!$e_2$ \\
\!\!$\vdots$
 \end{tabular}
 \end{center}
The GNLA structure is given by the rule: $\deg e_i=-i=\deg e_i'$.

For $m=n$ the structure equations change to:
 \begin{center}
 \begin{tabular}{c|cccccccccccccc}
\!\!& $e_1$ & $e_1'$ & $e_2$ & $e_3$ & $e_3'$ & $e_4$ & $e_4'$ & \!\!\dots\!\!
  & $e_m$ & $e_m'$& $e_{m+1}$ & $e_{m+1}'$ & $e_{m+2}'$\\
 \hline
\!\!$e_1$ & & $e_2$ & $e_3$ & $e_4$ & $e_4'$ & $e_5$ & $e_5'$ & \!\!\dots\!\!
 & $e_{m+1}$ & $e_{m+1}'$ & & $e_{m+2}'$\\
\!\!$e_1'$ & & & $e_3'$ & $e_4'$ & & $e_5'$ & & \!\!\dots\!\! & $e_{m+1}'$ & & $e_{m+2}'$\\
\!\!$e_2$ \\
\!\!$\vdots$
 \end{tabular}
 \end{center}
And for $n=m+1$ the table is:
 \begin{center}
 \begin{tabular}{c|ccccccccccccc}
\!\!& $e_1$ & $e_1'$ & $e_2$ & $e_3$ & $e_3'$ & $e_4$ & $e_4'$ & \!\!\dots\!\!
 & $e_{m+1}$ & $e_{m+1}'$ & $e_{m+2}$ & $e_{m+2}'$\\
 \hline
\!\!$e_1$ & & $e_2$ & $e_3$ & $e_4$ & $e_4'$ & $e_5$ & $e_5'$ & \!\!\dots\!\!
 & $e_{m+2}$ & $e_{m+2}'$\\
\!\!$e_1'$ & & & $e_3'$ & $e_4'$ & & $e_5'$ & & \!\!\dots\!\! & $e_{m+2}'$\\
\!\!$e_2$ \\
\!\!$\vdots$
 \end{tabular}
 \end{center}

  \end{prop}

In other words, $\mathfrak{f}_{m,n}$ is obtained through $m$ successive 1D
central extensions (this has sense to be called one central extension of order $m$)
of the Heisenberg algebra $\mathcal{H}_{n+2}=\langle e_1,e_1',e_2,e_3,\dots,e_{n+1}\rangle$.

 \begin{proof}
The generators are given by the formulae (we suppose $F_{z_nz_n}\ne0$;
we use the same symbol for the vector and its image under quotient to GNLA)
 \begin{multline*}
e_1=-\D_x,\ \ e_1'=\p_{z_n},\qquad e_i'=F_{z_nz_n}\p_{y_{m-i+2}}\quad (2<i\le m+2),\\
e_j=\p_{z_{n-j+1}}+F_{z_n}\p_{y_{m-j+1}}+c_j\p_{y_{m-j+2}}\quad (2\le j\le m+2),\\
e_k=\p_{z_{n-k+1}}\quad (m+2<k\le n+1)
 \end{multline*}
with coefficients $c_2=0$, $c_j=F_{z_{n-1}}+F_{z_n}F_{y_{m-1}}-(j-2)\D_x(F_{z_n})$ ($3\le j\le m+2$)
playing no role in what follows. For $n=m+1$ the last line in the above formula is void, and
for $n=m$ we shall additionally modify the range of the middle line ($2\le j\le m+1$).
The structure equations follow.
 \end{proof}

\subsection{Calculation of the Tanaka algebra}\label{S32}

Now let us describe the Tanaka algebra of equation (\ref{E1}) with $F_{z_nz_n}\ne0$.

 \begin{prop}
Prolongation $\mathfrak{t}_{m,n}=\widehat{\mathfrak{f}_{m,n}}$ of the GNLA $\mathfrak{f}_{m,n}$
has the following graded structure
(the number means the dimension while the subscript - the grading):
 $$
(1_{(-m-2)},2_{(-m-1)}\dots,2_{(-4)},2_{(-3)},1_{(-2)},2_{(-1)},2_{(0)})\quad\text{ for }n=m;
 $$
then for $n>m>1$:
 $$
(1_{(-n-1)}\dots1_{(-m-3)},2_{(-m-2)}\dots2_{(-3)},1_{(-2)},2_{(-1)},3_{(0)},1_{(1)}\dots1_{(n-m-1)});
 $$
and for $n>2$, $m=1$ we get:
 $$
(1_{(-n-1)},\dots,1_{(-4)},2_{(-3)},1_{(-2)},2_{(-1)},3_{(0)},2_{(1)},1_{(2)},\dots,1_{(n-2)}).
 $$
 \end{prop}

Let us notice that for $m=1,n=2$ the prolongation is different. Indeed, in this case
$\widehat{\mathfrak{f}_{1,2}}$ is isomorphic to the Lie algebra $\mathfrak{g}_2$ with the grading
 $$
(2_{(-3)},1_{(-2)},2_{(-1)},4_{(0)},2_{(1)},1_{(2)},2_{(3)}).
 $$
In the case $m=n=1$ the prolongation is infinite, corresponding to the contact algebra
of symmetries of $J^1(\R,\R)$:
 $$
(1_{(-2)},2_{(-1)},4_{(0)},6_{(1)},9_{(2)},12_{(3)},16_{(4)},20_{(5)},25_{(6)},30_{(7)},36_{(8)}\dots).
 $$

To be more precise the GNLA structure of $\mathfrak{t}_{m,n}=\oplus g_i$ is given by
the structure equations of Proposition \ref{L1} coupled with the following.

 \begin{prop}\label{L2}
For $n=m$: $\mathfrak{t}_{m,m}=\mathfrak{f}_{m,m}\oplus g_0$, with the space $g_0$ Abelian and generated
by $e_0^{10},e_0^{01}$ given by
 \begin{multline*}
[e_0^{ac},e_1]=a\,e_1,\ [e_0^{ac},e_1']=c\, e_1',\\
[e_0^{ac},e_k]=((k-1)a+c)\,e_k,\ k\ge2,\\
[e_0^{ac},e_k']=((k-2)a+2c)\,e_k',\ k>2.
 \end{multline*}

For $n>m>1$ the space $g_0\subset\mathfrak{t}_{m,n}$ is a Borel subalgebra of
$\op{gl}(g_{-1})$ and is generated by $e_0^{100},e_0^{010},e_0^{001}$ given by
 \begin{multline*}
[e_0^{abc},e_1]=a\,e_1+b\,e_1',\ [e_0^{abc},e_1']=c\, e_1',\\
[e_0^{abc},e_k]=((k-1)a+c)\,e_k+(k-2)b\,e_k',\ k\ge2,\\
[e_0^{abc},e_k']=((k-2)a+2c)\,e_k',\ k>2.
 \end{multline*}
Further on, $\mathfrak{t}_{m,n}=\mathfrak{f}_{m,n}\oplus g_0\oplus\dots\oplus g_{n-m-1}$,
where $g_i$ for $i>0$ are one-dimensional and generated by $e_{-i}$ given by (only
non-trivial relations are shown, and we omit commutators of non-negative
gradations\footnote{which restore uniquely by the derivation property.}):
 \begin{multline*}
[e_{-1},e_1]=e_0^{010},\ [e_{-1},e_k]=\tfrac{(k-2)(k-3)}2\,e_{k-1}',\\
[e_{-2},e_1]=e_{-1},\ [e_{-2},e_k]=\tfrac{(k-2)(k-3)(k-4)}6\,e_{k-2}',
\quad\dots \\
\dots\quad
[e_{m+1-n},e_1]=e_{m+2-n},\ [e_{m+1-n},e_k]=\tbinom{k-2}{n-m}\,e_{k+m+1-n}'
 \end{multline*}
(in the above formulae for $[e_{-p},e_k]$ we restrict $p+3\le k\le m+p+2$).

When $n>2$, $m=1$, the above formulae hold true with the only difference that $g_1$ is two-dimensional
and is generated by $e_{-1}^{10},e_{-1}^{01}$, so that we change $e_{-1}$ to $e_{-1}^{10}$ in the above,
and add the only nontrivial relations
 \begin{multline*}
[e_{-1}^{01},e_1]=e_0^{-2,0,1}=e_0^{001}-2\,e_0^{100},\
[e_{-1}^{01},e_1']=(1-n^2)\,e_0^{010},\\
[e_{-1}^{01},e_k]=(n^2-k^2+2k-4)\,e_{k-1}
 \end{multline*}
for the second generator of $g_1$.
 \end{prop}

 \begin{proof}
Consider at first the case $n=m$. Then $g_{-1}$ canonically splits as $\R e_1\oplus\R e_1'$, where
$e_1$ is given (up to scaling) by the relation $[e_1,e_{m+1}]=0$ (with $e_{m+1}=\op{ad}_{e_1}^{m-1}(e_2)$)
and $e_1'$ by the relation $[e_1',e_3']=0$ (with $e_3'=\op{ad}_{e_1'}(e_2)$).
Thus $g_0$ is diagonal, and in fact it equals the space $\{e_0^{ac}\}$ from the display formula of the Proposition.

To see that $g_1=0$, consider its element $e_{-1}$. It is uniquely given by 4 parameters:
 $$
[e_{-1},e_1]=e_0^{ac},\ [e_{-1},e_1']=e_0^{a'c'}.
 $$
This implies the further relations:
 \begin{multline*}
[e_{-1},e_2]=c\,e_1'-a'\,e_1,\
[e_{-1},e_3]=(a+2c)\,e_2,\
[e_{-1},e_3']=(2a'+c')\,e_2,\\
[e_{-1},e_4]=(3a+3c)\,e_3,\
[e_{-1},e_4']=(2a'+c')\,e_3+(a+2c)\,e_3',\\
[e_{-1},e_5]=(6a+4c)\,e_4,\ \dots
 \end{multline*}
Relation $0=[e_{-1},[e_1',e_3']]=(3a'+3c')\,e_3'$ yields $c'=-a'$.

There are two ways to evaluate $[e_{-1},e_5']$ since $e_5'=[e_1,e_4']=[e_1',e_4]$ and they do not match,
implying $a'=c'=0$, $c=0$. We continue calculation and get
 $$
[e_{-1},e_k]=\tfrac{(k-1)(k-2)}2a\,e_{k-1},\ [e_{-1},e_k']=\tfrac{(k-2)(k-3)}2a\,e_{k-1}'.
 $$
For $k=m+2$ (writing $[e_{-1},e_{m+1}]=0$ instead of $e_{m+2}$) we get $a=0$.

Let now $n>m>1$. Then only the line $\R e_1'\subset g_{-1}$ is canonical, and it's easy to check that $g_0$
is upper triangular in the basis $e_1,e_1'$, with general element $e_0^{abc}$ as in the display formula
in the Proposition.

Again to calculate $g_1$ we describe its element $e_{-1}$ via the formulae
 $$
[e_{-1},e_1]=e_0^{abc},\ [e_{-1},e_1']=e_0^{a'b'c'},
 $$
which imply
 \begin{multline*}
[e_{-1},e_2]=(c-b')\,e_1'-a'\,e_1,\\
[e_{-1},e_3]=(a+2c-b')\,e_2,\
[e_{-1},e_3']=(2a'+c')\,e_2,\\
[e_{-1},e_4]=(3a+3c-b')\,e_3+b\,e_3',\
[e_{-1},e_4']=(2a'+c')\,e_3+(a+2c)\,e_3',\\
[e_{-1},e_5]=(6a+4c-b')\,e_4+3b\,e_4',\ \dots
 \end{multline*}
Next $0=[e_{-1},[e_1',e_3']]=(3a'+3c')\,e_3'$ yields $c'=-a'$.
Evaluation of $[e_{-1},e_5']$ in two ways using $e_5'=[e_1,e_4']=[e_1',e_4]$ gives two different results,
whence $a'=c'=0$, $b'=c$.

The further calculations go without contradictions with the result
 \begin{multline*}
[e_{-1},e_k]=(\tfrac{(k-1)(k-2)}2a+(k-2)c)\,e_{k-1}+\tfrac{(k-2)(k-3)}2b\,e_{k-1}',\\
[e_{-1},e_k']=(\tfrac{(k-2)(k-3)}2a+2(k-3)c)\,e_{k-1}'.
 \end{multline*}
Now we need to specify and consider first $n=m+1$.
Substituting the last expression into $[e_{-1},[e_1,e_{m+2}]]=0$ and
$[e_{-1},[e_1,e_{m+2}']]=0$ we get 3 independent constraints implying $a=b=c=0$.
Thus in this case $g_1=0$.

Let now $n>m+1$. Then the relations are $[e_{-1},[e_1,e_{n+1}]]=0$ and
$[e_{-1},[e_1,e_{m+2}']]=0$ implying only $a=c=0$, thus making $e_{-1}$ with $b=1$ the
generator of $g_1$.

Next we consider $g_2$ with element $e_{-2}$ given by
 $$
[e_{-2},e_1]=\a\,e_{-1},\ [e_{-2},e_1']=\a'\,e_{-1}.
 $$
This yields
 \begin{multline*}
[e_{-2},e_2]=-\a'\,e_0^{010},\
[e_{-2},e_3]=\a'\,e_1',\
[e_{-2},e_3']=0,\\
[e_{-2},e_4]=\a'\,e_2,\
[e_{-2},e_4']=0,\
[e_{-2},e_5]=\a'\,e_3+\a\,e_3',\ \dots
 \end{multline*}
Evaluation of $[e_{-2},e_5']$ via $e_5'=[e_1,e_4']=[e_1',e_4]$ gives $2\a'\,e_3'$ and 0, whence $\a'=0$.
If $n>m+2$ there are no more contradictions and we get
 $$
[e_{-2},e_k]=\tfrac{(k-2)(k-3)(k-4)}6\,\a\,e_{k-2}',\  [e_{-2},e_k']=0,
 $$
i.e. $g_2$ is one-dimensional.

For $n=m+2$ the last relation in $\mathfrak{f}_{m,n}$ forces $\a=0$ and thus $g_2=0$ stops
the prolongation. Continuing in the same way we uncover the result for $n>m>1$.

The case $n>2$, $m=1$ is considered in the same manner giving us the required prolongation
$\widehat{\mathfrak{f}_{1,n}}$. Indeed, the $g_0=\{e_0^{abc}\}$  is the same
with relations ($\d_i^j$ is the Kronecker delta)
 \begin{multline*}
[e_0^{abc},e_1]=a\,e_1+b\,e_1',\ [e_0^{abc},e_1']=c\,e_1',\\
[e_0^{abc},e_k]=((k-1)a+c)\,e_k+\d_k^3\,b\,e_3'\ (k\ge2),\ [e_0^{abc},e_3']=(a+2c)\,e_3'.
 \end{multline*}
Next is $g_1$ the elements of which are specified by the relations
 $$
[e_{-1},e_1]=e_0^{abc},\ [e_{-1},e_1']=e_0^{a'b'c'}.
 $$
Similar to the above we obtain $a'=c'=0$, $a=-2c$, $b'=(1-n^2)c$.
Thus the general element is $e_{-1}^{bc}$ given by the only nontrivial relations:
 \begin{multline*}
[e_{-1}^{bc},e_1]=e_0^{-2c,b,c},\ [e_{-1}^{bc},e_1']=e_0^{0,(1-n^2)c,0},\\
[e_{-1}^{bc},e_k]=((k-1)(3-k)+n^2-1)c\,e_{k-1}+\d_k^4\,b\,e_3'.
 \end{multline*}
Next comes $g_2$ and the same arguments show that it is one-dimensional. Continuing we get
$g_p=\langle e_{-p}\rangle$ for $2\le p\le n-2$ with the only non-trivial relations
(between elements of positive and negative gradings):
 $$
[e_{-p},e_1]=e_{1-p},\ [e_{-p},e_{p+3}]=e_3'
 $$
(and we mean $e_{-1}^{10}$ instead of $e_{-1}$ in the very first formula for $p=2$).

Prolongation of $\mathfrak{f}_{1,1}$ is rather standard, and the case of $\mathfrak{f}_{1,2}$
is due to Tanaka \cite{T} (see also \cite{Y}).
 \end{proof}

 \begin{cor}\label{Cr2}
$\dim\mathfrak{t}_{m,n}=2n+4$ for $n\ge m>1$ and it equals $2n+5$ for $n>2,m=1$.
In exceptional cases $\dim\mathfrak{t}_{1,2}=14$ and $\dim\mathfrak{t}_{1,1}=\infty$.
 \end{cor}

Thus in non-exceptional cases $\dim\widehat{\mathfrak{f}_{m,n}}\le2\dim\mathfrak{f}_{m,n}-1$
with the equality only for $m=1$ (the deviation being $2m-1$ for $m>1$).

\subsection{Symmetric models}\label{Ss33}


In this section we shall calculate the internal infinitesimal symmetry algebra
$\op{sym}(\E_{m,n})$ for the model equation (\ref{E2}).
The upper bound on the number of independent symmetries is given by Corollary \ref{Cr2},
and we shall show that this upper bound is realized by an explicit construction.
An isomorphism between $\widehat{\frak{f}_{m,n}}$ and $\sym(\E_{m,n})$ together with algebraic
description of the latter will be given.

If $m>1$, the symmetry algebra has dimension $\dim\sym(\E_{m,n})=2n+4$.
These symmetries split into two types. The first type are the point symmetries with
the generators: shifts
 $$
S_0=\p_x,\ Y_i=\tfrac{x^i}{i!}\p_y\ (0\le i<m),\ Z_j=\tfrac{x^j}{j!}\p_z\ (0\le j<n)
 $$
and scalings
 $$
S_1=x\,\p_x+(m-1)\,y\,\p_y+(n-\tfrac12)\,z\,\p_z,\ R=y\,\p_y+\tfrac12z\,\p_z.
 $$
These form a subalgebra of symmetries of dimension $m+n+3$, which acts transitively on
$\E_{m,n}$ (of dimension $m+n+2$).
As we shall see in Proposition \ref{L3+}, the subalgebra
$\langle S_0,S_1,Y_i,Z_j\rangle\subset\sym(\E_{m,n})$ coincides with the Carnot
algebra $\frak{f}_{m,n}$, while the vector field $S_1+3R$ corresponds to
the \OE uler vector\footnote{this is the unique vector characterized by the condition
$[\epsilon,v]=k\cdot v$ $\forall$ $v\in g_k$.} $\epsilon\in g_0$.

Then there are higher symmetries with generators $(0\le k\le n-m)$:
 $$
Z_{n+k}=2\sum_{p=0}^k(-1)^p\tbinom{m+p-1}{p}\tfrac{x^{k-p}}{(k-p)!}\,z_{n-m-p}\,\p_y+\tfrac{x^{n+k}}{(n+k)!}\,\p_z
 $$
To verify that these are symmetries we prolong them to
 $$
\hat Z_{n+k}=2\sum_{i=0}^m\sum_{p=0}^k(-1)^p\tbinom{m+p-1-i}{p}\tfrac{x^{k-p}}{(k-p)!}\,z_{n-m-p+i}\,\p_{y_i}
+\sum_{j=0}^n\tfrac{x^{n+k-j}}{(n+k-j)!}\,\p_{z_j}
 $$
(with $\binom{p-1}p=\d_p^0$) and check that
 $$
\hat Z_{n+k}(y_m-z_n^2)=0.
 $$
 \com{
(the fact that these are symmetries and there are no other local one for different $k$ follows from the
lemma that $D_x^{-m}x^kD_x^n$ is a local operator for $k\le n-m$ only; exact formula for this operator
yields the form of finite symmetries).
 }%
When $n=m$, $Z_n$ is also a point symmetry\footnote{this explains why this case separates in Proposition \ref{L1}.};
otherwise $Z_{n+k}$ are genuine higher symmetries.

Thus the totality of symmetries is $2n+4$ and since by Theorem~\ref{ThT}
$\dim\sym(\E_{m,n})\le\dim\widehat{\frak{f}_{m,n}}=2n+4$, there are no more continuous symmetries.

For $m=1$, $n>2$ there is one additional symmetry\footnote{for $n=2$ this symmetry persists as well, but
there are 5 more symmetries.}, namely
 $$
S_2=x^2\p_x+(nz_{n-1})^2\p_y+(2n-1)xz\,\p_z,
 $$
and then we have\footnote{with $\dim\E_{1,n}=n+3$, so that $2(n+3)-1=\dim ST^*M$, cf. \cite{DZ}.}
 $\dim\sym(\E_{1,n})=2(n+3)-1$. Thus we have proved:

 \begin{prop}\label{L3}
The algebra $\sym(\E_{m,n})$ of internal symmetries of equation (\ref{E2})
has dimension $2n+4$ for $n\ge m>1$ and dimension $2n+5$ for $n>2,m=1$
with the basis $S_l,R,Y_i,Z_l$.
 \end{prop}

Let us describe the Lie algebra structure for $\sym(\E_{m,n})$.
Assume at first $m>1$ and let the index $l=0,\dots,2n-m$.
The non-zero commutators are the
following (with $Y_{-1}=Z_{-1}=0$).
 \begin{multline*}
[S_0,S_1]=S_0,\ [S_0,Y_i]=Y_{i-1},\ [S_0,Z_l]=Z_{l-1},\\
[S_1,Y_i]=(i+1-m)Y_i,\ [S_1,Z_l]=(l+\tfrac12-n)Z_l,\\
[Y_i,R]=Y_i,\ [Z_l,R]=\tfrac12 Z_l.
 \end{multline*}
These are accompanied by splitting $Z_l$ into $Z_j$ ($0\le j<n$) and $Z_{n+k}$ ($0\le k\le n-m$)
and the relations:
 $$
[Z_j,Z_{n+k}]=2(-1)^k\tbinom{n-j-1}k Y_{j-n+m+k},\quad n-m-k\le j<n-k.
 $$

For $m=1$ (when $Y_i$ are reduced to solely $Y_0$) we add the following commutators for $S_2$
 $$
[S_0,S_2]=2S_1,\ [S_1,S_2]=S_2,\ [S_2,Z_l]=((l+1-n)^2-n^2)Z_{l+1}.
 $$

By Theorem~\ref{ThT} and Proposition \ref{L3} the symmetry algebra of our model
is isomorphic to the Tanaka algebra $\frak{t}_{m,n}=\widehat{\frak{f}_{m,n}}$. We make this explicit.

 \begin{prop}\label{L3+}
The isomorphism $\sym(\E_{m,n})\simeq\frak{t}_{m,n}$ is given by the following identifications:
 \begin{multline*}
S_0=e_1,\ S_1=\tfrac12\,e_0^{001}-e_0^{100},\ S_2=-e_{-1}^{01},\ R=-\tfrac12\,e_0^{001},\\
Y_i=e_{m+2-i}'\ (0\le i<m),\ Z_j=e_{n+1-j}\ (0\le j<n),\ Z_n=-2\,e_1',\\
Z_{n+1}=2\,e_0^{010},\ Z_{n+2}=-2\,e_{-1}^{10},\ Z_{n+k+1}=2(-1)^ke_{-k}\ (1<k<n-m).
 \end{multline*}

For $m>1$ we shall remove $S_2$ and $e_{-1}^{01}$ with change $e_{-1}^{10}\mapsto e_{-1}$.
And for $m=n$ we identify $e_0^{abc}\mapsto e_0^{ac}$ and remove $e_0^{010}$ ($Z_{n+1}$ is non-existent
in this case).
 \end{prop}

In particular, we obtain the weights giving gradation of $\sym(\E_{m,n})$:
 $$
w(S_k)=k-1,\ w(R)=0,\ w(Y_i)=i-m-2,\ w(Z_j)=j-n-1.
 $$

For $m=1$, the radical is $\mathfrak{r}=\langle Y_0,Z_l,R\rangle_{0\le l\le 2n-1}$ and the vectors
$S_0,S_1,S_2$ define an $sl_2$ complement. The length of the derived series for $\frak{r}$
equals 3. The nilradical is $\langle Y_0,Z_l\rangle_{0\le l<2n}$ and its class
(the length of lower central series) equals 2.

For $m>1$ the symmetry algebra is solvable.
The length of the derived series equals 3 for $n=m$ and 4 for $n>m$.
The nilradical is $\langle S_0,Y_i,Z_l\rangle_{0\le i<m;0\le l\le 2n-m}$ and has class
equal $2n-m+1$.


\subsection{End of the proof and a generalization}\label{S34}

Now we can prove our main results.

\smallskip

 \begin{Proof}{Proof of Theorem \ref{Thm1}}
Let $\Delta$ be the distribution corresponding to equation (\ref{E1})
with non-degeneracy condition $F_{z_nz_n}\ne0$ satisfied and $\C$ the one for
equation (\ref{E2}). We have demonstrated inequality $\dim\sym(\Delta)\le\dim\sym(\C)$.

Moreover by virtue of Theorem \ref{ThT}
the equality
 $$
\dim\sym(\Delta)=\dim\sym(\C)=\dim\frak{t}_{m,n}
 $$
implies $\sym(\Delta)\simeq\frak{t}_{m,n}$
and the distribution $\Delta$ of the equation $\E$ is flat.

Therefore it is locally isomorphic to the standard model, which coincides by the above
with the distribution $\C$. Thus the equation $\E$ is locally internally equivalent to 
the equation $\E_{m,n}$ (\ref{E2}).
\qed
 \end{Proof}

\medskip

Consider Monge ODEs (\ref{E1})
with non-degeneracy condition $F_{z_nz_n}\ne0$ (sufficient for $\dim\sym(\Delta)<\infty$).
They correspond internally to rank 2 distributions on a manifold of $\dim(\E)=m+n+2=k$.
Denote the class of all such distributions with $m\ge\mu$ and fixed $k$ by $\mathfrak{D}(\mu,k)$.
Our calculations imply:

 \begin{cor}
Among distributions in the class $\mathfrak{D}(m,k)$ the most symmetric
are locally equivalent to the distribution of equation (\ref{E2}).
 \end{cor}
This implies in turn our Corollary \ref{Cor1}. Thus if we denote $\frak{p}_{n+3}=\frak{f}_{1,n}$,
then the largest algebra of symmetries of distributions in $\mathfrak{D}(1,k)$ is $\widehat{\frak{p}_k}$.

\smallskip

{\bf Remarks:} {\bf 1.} Any rank 2 distribution can be realized as the Cartan distribution of an underdetermined
ODE system. By the same trick as in D-modules this can be rewritten as a single higher order ODE.
The equation manifold has however higher dimension and the former distribution is obtained from the
latter by a sequence of de-prolongations. Thus without non-degeneracy condition Monge equations are
internally indistinguishable from the general rank 2 distributions.

{\bf 2.} Dimension 5 admits only one Lie algebra $\frak{p}_5$ of the type $(2,1,2)$.
In dimension 6 there are already 3 GNLA of $(2,1,2,1)$-type: $\frak{p}_6=\frak{f}_{1,3}$, $\frak{h}_6=\frak{f}_{2,2}$,
and one more $\frak{ell}_6$ not realizable as a distribution of a nondegenerate single Monge ODE \cite{AK}.
There is a reason to call them parabolic, hyperbolic and elliptic, whence our choice of the letters.

{\bf 3.} In higher dimensions the number of GNLAs grows unboundedly, and it is impossible to trace the
classification. Study of low-dimensional cases and some other types of underdetermined ODEs leads to the following
claim, which will be discussed in a forthcoming paper \cite{AK}:

\medskip

{\bf Conjecture:} {\em Equation $\E_{1,k-3}$ (\ref{E2}) represents the most symmetric distribution $\C$
among all rank 2 distributions $\Delta$ on a manifold $\E$ of $\dim(\E)=k$
with $\dim\sym(\Delta)<\infty$.\/}

\section{Lie-B\"acklund type theorem}\label{S4}

Consider a bundle $\pi:E\to M$. A Lie transformation of the jet space $J^k\pi$ is, by definition,
a diffeomorphism preserving the canonical Cartan distribution $\C=\C^k$
($k$ is not the rank, but we write it as a superscript to avoid confusions).
The classical Lie-B\"acklund theorem states that a Lie transformation $F$
of the jet space $J^k\pi$ is the lift of a Lie transformation $f$ of $J^\epsilon\pi$,
where $\epsilon=1$ and $f$ is a contact transformation of $J^1\pi$ in the case $\op{rank}\pi=1$,
while $\epsilon=0$ and $f$ is a (local) diffeomorphism of $J^0\pi$ (point transformation)
for $\op{rank}\pi>1$.

Various generalizations of this to Lie transformations of equations $\E\subset J^k\pi$ are known
\cite{KLV,AKO,GK}. In this section we present a generalization to the case of Monge equations (\ref{E1}).

\subsection{Lie transformations of the mixed jets}\label{Ss41}

To prove our main result we first consider the case of mixed
jets\footnote{The results of this subsection have a straightforward generalization to the case of
more general mixed jets $J^{k_1,\dots,k_s}$.}
 $$
J^{k_1,k_2}(M,N_1\times N_2)=J^{k_1}(M,N_1)\times_MJ^{k_2}(M,N_2)
 $$
with points $(x,j^{k_1}_x(s_1),j^{k_2}_x(s_2))$, where
$j^{k_i}_x(s_i)$ denotes $k_i$-jet of the map $s_i:M\to N_i$ at the point $x\in M$.
There are natural projections
 \begin{gather*}
\pi_i:J^{k_1,k_2}(M,N_1\times N_2)\to J^{k_i}(M,N_i)\qquad (i=1,2),\\
\pi^j:J^{k_1,k_2}(M,N_1\times N_2)\to J^{k_1-j,k_2-j}(M,N_1\times N_2).
 \end{gather*}
The Cartan distributions on the space of pure jets induce the one on the mixed jets by the formula
 $$
\C^{k_1,k_2}=(d\pi_1)^{-1}(\C^{k_1})\cap(d\pi_2)^{-1}(\C^{k_2})\subset TJ^{k_1,k_2}.
 $$
Its (weak=strong) derived flag is
 $$
(\C^{k_1,k_2})_j=(d\pi^j)^{-1}(\C^{k_1,k_2}).
 $$
Note that in the case considered in the next section, namely $\dim M=\dim N_1=\dim N_2=1$,
the growth vector of $\C^{k_1,k_2}$ is $(3,2,\dots,2,1,\dots,1)$.

To exclude the known case $J^k(M,N_1\times N_2)=J^{k,k}(M,N_1\times N_2)$ of pure jets,
we assume in the next statement that $k_1<k_2$.

 \begin{theorem}\label{LBthEmpt}
1. Every Lie transformation $F$ of the mixed jet space $J^{k_1,k_2}(M,N_1\times N_2)$ is the prolongation
$F_0^{(k_1)}$ of a Lie transformation $F_0$ of the space $J^{0,k_2-k_1}(M,N_1\times N_2)$.

2. Under identification $J^{0,k}(M,N_1\times N_2)=N_1\times J^k(M,N_2)$ (for $k=k_2-k_1$)
every Lie transformation $F_0$ is given by the formula
 $$
F_0\bigl(y,j_x^k(s)\bigr)=\bigl(Y(y,j_x^k(s)),f^{(k-\epsilon)}(j_x^k(s))\bigr),
 $$
where $\epsilon=1$ and $f:J^1(M,N_2)\to J^1(M,N_2)$ is a contact transformation
in the case $\dim N_2=1$, while $\epsilon=0$ and $f^{(k)}$ is the lift of a point transformation
$f:M\times N_2\to M\times N_2$ for $\dim N_2>1$.
 \end{theorem}

 \begin{proof}
For \textit{1} we mimic the proof of the classical Lie-B\"acklund theorem:
from the derived flag of $\C^{k_1,k_2}$ we deduce that the Cauchy characteristic space of
$(\C^{k_1,k_2})_{j+1}$ is $\op{Ker}(d\pi^j)$ and thus every Lie transformation of $J^{k_1,k_2}$ fibers
over a Lie transformation of $J^{k_1-j,k_2-j}$ for any $1\le j\le k_1$.
Standard arguments show that the former transformation is the prolongation of the latter.

To prove \textit{2} let us notice that the Cauchy characteristics space of the Cartan distribution $\C^{0,k}$
equals $TN_1=\op{Ker}(d\pi_2)$. Thus every Lie transformation of $J^{0,k}(M,N_1\times N_2)$
fibers over a Lie transformation of the pure jet space $J^k(M,N_2)$ and an
application of the classical Lie-B\"acklund theorem finishes the proof.
 \end{proof}

\subsection{Lie transformations of Monge ODEs}
Internal equivalence of two equations $\E,\E'$ of type (\ref{E1}) is a diffeomorphism
$\phi:\E\to\E'$ which maps the induced Cartan distributions $\C_\E$ to $\C_{\E'}$.
Proposition \ref{L1} implies that two nondegenerate Monge ODEs with different pairs
$(m,n)$ are never internally equivalent.

External equivalence of two Monge equations $\E,\E'\subset J^{m,n}(\R,\R\times\R)$ is
a Lie transformation of $J^{m,n}(\R,\R\times\R)$ mapping $\E$ to $\E'$.
In particular, it induces an internal equivalence.


 \begin{theorem}
Every internal equivalence of two non-degenerate Monge equations (\ref{E1}) with
$n\ge m>0$ and $(m,n)\ne(1,1),(1,2)$ is induced by a unique external Lie transformation.
 \end{theorem}

Contrary to other Lie-B\"acklund type theorems \cite{KLV,GK}, it is not enough here
to study Cauchy characteristics of the derived distributions or characteristic subspaces.
We shall use the Tanaka theory.

 \begin{proof}
Let $\C$ be the Cartan distribution of (\ref{E1}). Its
(weak=strong) derived flag $\C_i$ was calculated in Section \ref{S31}.

According to the proof of Proposition \ref{L2}, the Cartan distribution $\C$ of (\ref{E1})
contains the canonical (from internal viewpoint) line $\mathcal{L}=\langle e_1'\rangle$.
Namely the Carnot algebra $\mathfrak{f}_{m,n}$ contains precisely one vector $v=e_1'\in g_{-1}$
(up to scaling) satisfying $\op{ad}_v^3=0$. From external viewpoint this line $\mathcal{L}$
is generated by the vertical vector (in internal coordinates $\p_{z_n}$, in external
$\p_{z_n}-F_{z_n}\p_{y_m}$).
The quotient of $\E$ by the corresponding foliation can be naturally identified with $J^{m-1,n-1}(\R,\R^2)$.
Indeed, if we denote by $j:\E\to J^{m,n}$ the inclusion and by $\pi:\E\to\E/\mathcal{L}$ the projection,
then for $\pi^1:J^{m,n}\to J^{m-1,n-1}$ we have: $\pi=\pi^1\circ j$.

Furthermore, the distribution
$\mathcal{H}=\langle\D_x,\p_{z_n},\p_{z_{n-1}},\p_{y_{m-1}}\rangle\subset\C_3$
is also canonical (intrinsically defined) as it equals $\C_2\oplus[e_1',\C_2]$
(in the Carnot algebra it corresponds to $\g_{-1}\oplus\g_{-2}\oplus\langle e_3'\rangle$).
In the quotient of $\E$ by $\mathcal{L}$ the distribution $\mathcal{H}$ goes to the Cartan
distribution of the jet space, $d\pi(\mathcal{H})=\C^{m-1,n-1}$ .

Consider two Monge ODEs $\E,\E'\subset J^{m,n}(\R,\R^2)$ and an internal equivalence $\Psi:\E\to\E'$.
It satisfies $\Psi_*(\mathcal{L})=\mathcal{L}'$ and $\Psi_*(\mathcal{H})=\mathcal{H}'$
and thus induces a Lie transformation $\Phi$ of the jet-space, i.e. we have the following
commutative diagram:
 $$\begin{CD}
\E @>{\Psi}>> \E' \\ @V{\pi}VV @VV{\pi'}V \\ J^{m-1,n-1}(\R,\R^2) @>{\Phi}>> J^{m-1,n-1}(\R,\R^2).
 \end{CD}$$
By Theorem \ref{LBthEmpt}, $\Phi$ equals the prolongation of a Lie transformation
$\phi$ of $J^{0,n-m}(\R,\R^2)$.
To complete the proof we must show that $\Psi=\phi^{(n-m)}|_\E$.

First, we claim that $\phi^{(n-m)}$ maps $\E$ to $\E'$.
Indeed, a point $b\in\E$ is represented by the class of tangent spaces $T_{\pi(b)}(\pi(L^1))$,
where $L^1$ runs through integral curves of $(\E,\C)$ passing through $b$.
Thus $\Psi(b)$ is determined by the collection of spaces $\Phi_*T_{\pi(b)}(\pi(L^1))$,
which coincides with
 $$
T_{\Phi(\pi(b))}(\phi^{(n-m-1)}(\pi(L^1)))=T_{\pi'(\phi^{(n-m)}(b))}(\pi'(\phi^{(n-m)}(L^1)))
 $$
due to equality $\Phi=\phi^{(n-m-1)}$.
This implies that $\phi^{(n-m)}(\E)=\E'$, and let us denote
 $$
I=\Psi^{-1}\circ\phi^{(n-m)}:\E\to\E.
 $$
This map is a Lie transformation fibered over the identity map of $J^{m-1,n-1}(\R,\R^2)$.
Writing conservation of the Cartan form $dz_{n-1}-z_n\,dx$ we obtain that $I$ is indeed the identity.
 \end{proof}

 \com{
By Proposition \ref{L1}
$\langle e_1',e_2,e_3',e_4'\dots,e_{i+2}'\rangle\subset\C_{i+2}$ is a canonical sub-distribution and
therefore we can canonically define the following distribution:
 $$
\bar\C_{i+2}=\C_{i+1}\oplus\langle e_{i+2}'\rangle =
\langle \D_x,\p_{y_{m-1}},\dots,\p_{y_{m-i}},\p_{z_n},\dots,\p_{z_{n-i}}\rangle \subset\C_{i+2}
 $$
(we use formulae from the proof of Proposition \ref{L1}: $e_{i+2}'=\p_{y_{m-i}}$ etc).

The Cauchy characteristic subspace of the latter is
 $$
\op{Ch}(\bar\C_{i+2})= \langle \p_{y_{m-1}},\dots,\p_{y_{m-i+1}},\p_{z_n},\dots,\p_{z_{n-i+1}}\rangle \subset\bar\C_{i+2}.
 $$
Quotient by it recovers the structure of the space $J^{m-i+1,n-i+1}(\R,\R^2)$.
 }

 \begin{rk}
The exact formulae from Section \ref{Ss33} show that symmetries of (\ref{E2}) are induced
by exterior transformations of jets $J^{m,n}$ as the theorem claims.
They have block form generalizing the feedback transformations in control theory
(with variable $z$, control $y$ and time $x$).
 \end{rk}

Let us discuss the exceptional cases. When $m=n=1$, the internal symmetry group of $\E:y_1=f(x,y,z,z_1)$
is isomorphic to the contact group of $J^1(\R,\R)$ and is parametrized by a single function in 3 variables.
On the other hand, the external symmetries can be calculated directly since, by the classical Lie-B\"acklund
theorem, they are prolongations of point symmetries of $J^0(\R,\R^2)$.
In particular, for the model equation $\E_{1,1}:y'=(z')^2$ the group of point symmetries
is 10-dimensional, and thus is much smaller than the infinite-dimensional internal group.

When $m=1$, $n=2$, the internal geometry of the equation $\E:y_1=f(x,y,z,z_1,z_2)$ was
described by E.\,Cartan \cite{C1}. The internal symmetry group is at most 14-dimensional,
and is realized by the equation $\E_{1,2}:y'=(z'')^2$ (see \cite{AKO,AF} for the
explicit generators). The symmetry group is the split real form of the exceptional Lie group $G_2$.
This group has a 9-dimensional subgroup of external symmetries coming from Lie
transformations of $J^{1,2}(\R,\R^2)$. But the remaining 5 generators are not induced by the external symmetries.
Thus, again in this case, the internal geometry is richer than the external geometry.

\section{Integrable extensions}\label{S5}

In naive terms, a covering $\tilde\E$ or an integrable extension \cite{KV,BG} of a general partial
differential equation $\E$ is given by introducing an additional dependent variable $v$
and relations $D(v)=F[x,u,v]$, where $F$ are vector differential operators on
the vector-functions $u=u(x)$, $v=v(x)$,
such that the expanded system is compatible by virtue of $\E$.
In this case solutions of the new system $\tilde\E$ can be found from the solutions of $\E$
by ODE methods.

In the case of (systems of) underdetermined ODEs with $u=(y,z,\dots)$ the additional equation has the form
 $$
v'=f(x,u,u',\dots,v)
 $$
and no integrability constraints arise.

In this section we prove that flat integrable extensions of flat model equations are in bijective
correspondence with central extensions of the corresponding Carnot algebras $\m$.

\subsection{Central extensions of GNLAs}

Any grading preserving automorphism of a fundamental GNLA is generated by its action on $g_{-1}$,
and therefore can be identified with a subgroup $\frak{H}\subset\op{GL}(g_{-1})$. Motivated by
our interest in central extensions, we call it the \emph{gauge group} of $\m$.
The Lie algebra $\frak{h}=\op{Lie}(\frak{H})$ equals the subspace $g_0\subset\g=\widehat\m$
of degree 0 derivations of $\m$.

A $d$-dimensional \emph{central extension} $\widetilde{\m}$ of a Lie algebra $\m$ is defined by the exact sequence
 \begin{equation}\label{CextES}
0\to\mathfrak{a}\hookrightarrow \widetilde{\m}\to\m\to0,
 \end{equation}
where $\mathfrak{a}$ is a $d$-dimensional Abelian algebra belonging to the center of $\widetilde{\m}$.
We are interested in the case, when both $\m$ and $\widetilde{\m}$ are fundamental GNLA,
and $\mathfrak{a}$ has a pure grading $-k$, i.e. $\mathfrak{a}\subset g_{-k}$.

For our applications the initial grading of GNLA is $(2,1,2,\dots)$, and then we have $3<k\le m+1$,
where $m$ is the length of $\m$. An extension of maximal grading $k=m+1$ is
 \begin{equation*}
\widetilde\m=g_{-m-1}\oplus (g_{-m}\oplus\dots\oplus g_{-1})=\mathfrak{a}\oplus\m
 \end{equation*}
(direct sum in the sense of vector spaces, but not Lie algebras).

It is well-known (see e.g. \cite{SL,F}) that the 1-dimensional non-trivial central extensions of a Lie
algebra $\m$ are enumerated by the elements of the Lie algebra cohomology $H^2(\m)$.
Zero element of the cohomology group corresponds to the trivial extension $\widetilde\m=\m\oplus\R$
and we will exclude such cases, which for exact sequence (\ref{CextES}) means that it splits.

The same proof shows that for a GNLA $\m$ central extensions of grading $-k$ correspond to the graded part of
the cohomology $H^2_k(\m)$. More generally, $d$-dimensional \emph{nontrivial extensions} are given
by $d$-dimensional subspaces in $H^2_k$, that is, elements of the Grassmanian $\op{Gr}_dH^2_k$.
Isomorphism classes of central extensions correspond to orbits of the gauge group action on $\op{Gr}_dH^2_k$,
so that applying an automorphism to the cohomology class leads to an equivalent central extension.

Thus we have proved

 \begin{prop}\label{prop14}
Let $\m$ be a GNLA with fundamental space $g_{-1}$ and the gauge group $\mathfrak{H}$.
Its $d$-dimensional nontrivial central extensions of grading $k$ are in bijective correspondence with
$\op{Gr}_dH^2_k(\m)/\mathfrak{H}$.
 \end{prop}

Let us notice that the space $g_0$ is always non-zero, since $\mathfrak{H}$
contains the scaling corresponding to multiplication by $\epsilon^k$ on $g_{-k}$.
However this scaling acts trivially on the Grassman space $\op{Gr}_dH^2_k(\m)$.

\subsection{Integrable extensions of flat models}

For general distributions $\Delta$ an \emph{integrable extension} (or covering) is a
submersion
 $$
\pi:(\tilde M,\tilde\Delta)\to(M,\Delta),\qquad\text{ such that }
 $$
 \begin{itemize}
\item[(i)] $d_x\pi:\tilde\Delta\stackrel{\sim}\to\Delta$ is an isomorphism for any $x\in\tilde M$,
\item[(ii)] for the derived distribution we have $\tilde\Delta_2\cap\op{Ker}d\pi=0$.
 \end{itemize}

Then by the Frobenius theorem, any integral manifold $N\subset(M,\Delta)$ lifts to
an integral manifold $\tilde N\subset(\tilde M,\tilde\Delta)$ passing through a given point.
Indeed, the induced distribution $\tilde\Delta\cap T\Sigma_N$ on
$\Sigma_N=\pi^{-1}N\subset\tilde M$ is Frobenius-integrable (and thus
its integration is reduced to ODEs).

 \begin{rk}
To recast this into the terminology of \cite{BG} let us choose a general connection
on the bundle $\pi$, i.e. a distribution of horizontal spaces $H\subset T\tilde M$
such that $\tilde\Delta_2\subset H$. This $H$ defines a vertical Pfaffian system
$J\subset\Omega^1(\tilde M)$ and the integrable extension is given by
$\tilde I=\op{Ann}(\tilde\Delta)$, $I=\op{Ann}(\Delta)$, related by $\tilde I=\pi^*I+J$.
 \end{rk}

In order to relate integrable extensions to central extensions, we restrict to
integrable extensions for which tangents to the fiber $V$ of $\pi$
have pure grading $-k$ in the Carnot algebra $\widetilde\m$. The integrability condition
$\tilde\Delta_2\cap\op{Ker}d\pi=0$ is then equivalent to the requirement $k>2$.

If $\dim\Delta=2$, then the integrability condition is vacuous, since we require the
growth vector starts $(2,1,2,\dots)$. In this case $k>3$.

As usual we assume that the distribution $\Delta$ is completely non-holono\-mic. An
integrable extension $\pi$ will be called \emph{non-trivial} if $\tilde\Delta$ enjoys the same property,
which requires additionally only that iterated brackets span $TV$ (and surely only in this case the above
grading $-k$ is defined).

 \begin{theorem}
1. Let $\pi:(\tilde M,\tilde\Delta)\to(M,\Delta)$ be a non-trivial integrable extension
with $d$-dimensional fibers $V$. Let $\tilde\m$ denote the Carnot algebra at
$x\in\tilde M$ and $\m$ the Carnot algebra at $y=\pi(x)\in M$.
Suppose that $\mathfrak{v}=T_xV$ correspond to pure grading $-k$ in $\tilde\m$, $k>2$.
Then $\tilde\m$ is a $d$-dimensional central extension of $\m$ (in grading $-k$).

2. Conversely, let $\tilde\m$ be a nontrivial central $d$-dimensional extension of a
fundamental GNLA $\m$ in pure grading $-k$, $k>2$, and let $(\tilde M,\tilde\Delta)$ and $(M,\Delta)$
be the corresponding flat model distributions.
Then $(\tilde M,\tilde\Delta)$ is an integrable extension of $(M,\Delta)$.
 \end{theorem}

 \com{
The assumption about pure grading is necessary, but can be overcome
since any extension (central or integrable) is a result of successive 1-dimensional
extensions (for $d=1$ the assumption is void).
 }

 \begin{rk}
In other words there is a bijective correspondence between classes
(modulo diffeomorphisms) of integrable flat extensions of flat distributions $(M,\Delta)$
and classes (modulo automorphisms of GNLA) of central extensions of $\m$
(nontrivial and of pure grading in both cases).
 \end{rk}

Before we start proving this theorem, let us establish the functoriality of the Tanaka map $\Delta\mapsto\m$.

 \begin{prop}\label{functoriality}
Any smooth map $\pi:(\tilde M,\tilde\Delta)\to(M,\Delta)$ induces a morphism of GNLAs
$\tilde\m_x\to\m_{\pi(x)}$.
 \end{prop}

Indeed, $d\pi$ maps the derived flags $\tilde\Delta_i\to\Delta_i$ and hence
 $$
d\pi:{\tilde g}_{-i}=\tilde\Delta_i/\tilde\Delta_{i-1}\to\Delta_i/\Delta_{i-1}=g_{-i}.
 $$
Since $\pi_*=d\pi$ respects commutators, the map is a homomorphism.

Now we can turn to the theorem about extensions.

 \begin{proof}
\textit{1.} Due to Proposition \ref{functoriality} we have the following exact sequence of vector spaces
 $$
0\to\mathfrak{v}\to\tilde\m\stackrel{\pi_*}\longrightarrow\m\to0.
 $$
Since $\pi_*$ is an epimorphism of GNLAs it suffices to show that $\mathfrak{v}$
belongs to the center of $\tilde\m$.

Let $u\in \tilde g_{-l}$ and $v\in\mathfrak{v}$. The assumption $\mathfrak{v}\subset \tilde g_{-k}$
implies $[u,v]\in \tilde g_{-l-k}$.
On the other hand, $[u,v]$ equals $[\tilde\x,\e]_x\!\mod\tilde\Delta_{k+l-1}$, where
$\tilde\x\in\Gamma(\tilde\Delta_l)$ is a vector field with value $u$ at $x$
and $\e\in\Gamma(V)$ is the vertical field with value $v$ at $x$.
Choosing a $\pi$-projectible vector field $\tilde\x$ over some $\x\in\Gamma(\Delta_l)$,
we have $[\tilde\x,\e]\in\Gamma(V)$.
This implies $[u,v]\in\mathfrak{v}\cap\tilde g_{-l-k}=0$.

\smallskip

\textit{2.} Consider a $d$-dimensional central extension
 $$
0\to V\to\widetilde\m\stackrel{\varpi}\longrightarrow\m\to0.
 $$
By the Tanaka theory the flat distribution $\Delta$ can be locally identified with
the standard distribution on the Lie group $M$ corresponding to $\m$, see \cite{T}.
Let $(\tilde M,\tilde\Delta)$ be the corresponding pair for $\widetilde\m$.
We take both $\tilde M=M\times V$ and $M$ simply-connected, and as nilpotent groups
they are topological vector spaces.

The distributions $\Delta$ and $\tilde\Delta$ are spanned by the left-invariant vector fields
from $g_{-1}$ and $\tilde g_{-1}$ respectively.
Define the submersion $\pi:\tilde M\to M$ to be the group homomorphism,
induced by the Lie algebras homomorphism $\varpi$. Then $d\pi$ maps
$\tilde\Delta$ onto $\Delta$ and $\tilde\Delta_2$ onto $\Delta_2$. Since $k>2$,
the distribution $\tilde\Delta_2$ is horizontal and hence $\pi$ is an integrable extension.

\smallskip

\textit{2'.} Let us give an alternative more constructive proof.

Consider the integrable extension given by linearly independent cohomology classes
$[\a_1],\dots,[\a_d]$, where the closed 2-forms $\a_i\in\La^2\m^*$ have grading $k>2$.
These forms are not exact on $\m$, but $\varpi^*\a_i$ are exact on the extension $\tilde\m$,
whence $\varpi^*\a_i=d\b_i$ and $\b_i|_V$ is a basis of $V^*$.

This construction can be made explicit as follows. Realize $\a_i$ as left-invariant
2-forms on $M$. Since $M$ is contractible, $\a_i=d\z_i\in\Omega^2(M)$. We can choose
$\z_i$ and linear coordinates $v_i$ on $V$ in such a way that the forms $\b_i=dv_i-\pi^*\z_i$
are left invariant on $\tilde M$.

Let $\tilde\Delta=\op{Ann}\langle\b_1,\dots,\b_d\rangle\cap\pi_*^{-1}(\Delta)$. We claim that
$\tilde\Delta_2$ does not intersect the fiber $V$.
Indeed, choose $\tilde v,\tilde w\in\Gamma(\tilde\Delta)$.
Without restriction of generality we can assume that $\tilde v,\tilde w$ cover sections $v,w$ of $\Delta$.
Then $\a_i(v,w)=0$ thanks to the assumption $k>2$, and so
$d\b_i(\tilde w,\tilde v)=\b_i([\tilde v,\tilde w])=0$.
Since $\tilde\Delta_2$ is generated by the commutators $[\tilde v,\tilde w]$, it is horizontal.
Therefore $\pi$ is an integrable extension.

Changing 2-forms within the cohomology class or applying an automorphism of $\m$ changes
the distributions $\tilde\Delta$ to equivalent.
 \end{proof}

 \begin{theorem}\label{thm17}
Consider a flat underdetermined ODE $\E:F[x,u]=0$ ($u=u(x)$ is a vector function and $F$ a vector differential operator)
corresponding to a rank 2 distribution $\Delta$ with Carnot algebra $\m$. Let $\widetilde\m$ be
a nontrivial central extension of the GNLA $\m$ of grading $k>3$.
Then its integrable extension can be realized as an underdetermined ODE system
($v$ has the same dimension as the extension)
 $$
F[x,u]=0,\ v'=G[x,u]
 $$
 \end{theorem}

In particular, if $\m=\mathfrak{f}_{m,n}$ and $\E=\E_{m,n}$ is the model Monge equation,
then the integrable extension can be realized as a pair of ODEs (which, in some cases, can be
reduced to a higher order Monge equation).

 \begin{proof}
To keep the notations simple we restrict to the case of a flat Monge equation.
In Section \ref{S31} we constructed coordinates $(x,y_i,z_j)$ on $\E$
(which are not the canonical coordinates on $M=\E$ considered as the nilpotent
Lie group corresponding to $\m=\mathfrak{f}_{m,n}$). On the integrable extension $\tilde M$
let us use the coordinates adapted for the projection $\pi$, that is the coordinates of $M$
together with vertical coordinates $v_l$.

Let $\oo_1,\oo_1',\oo_2,\dots$ be the left-invariant 1-forms dual to the basis
$e_1,e_1',e_2,\dots$ on $\m$ chosen in Proposition \ref{L1}. They satisfy
the Maurer-Cartan equations $d\oo_1=0$, $d\oo_1'=0$, $d\oo_2=\oo_1'\we\oo_1$, \dots
Notice that $\oo_1=-dx,\oo_1'=dz_n$ and $\Delta=\op{Ann}\langle\oo_i,\oo_j':i,j\ge2\rangle$.

Thus we can express the forms defining the extension (as in the previous proof) via the co-basis:
 $$
\b_l=dv_l-c_{l1}\oo_1-c_{l1'}\oo_1'-c_{l2}\oo_2-\dots,
 $$
where the coefficients $c_{l\z}$ are functions on $M$. By changing the 1-forms to cohomologous
$\b_l\mapsto\b_l-d[\int c_{l1'}\,dz_n]$ we can assume that the coefficient $c_{l1'}$ is zero.
Thus $\b_l=dv_l-g_ldx\!\mod\langle\oo_i,\oo_j':i,j\ge2\rangle$.

This implies that the integral extension
$\tilde\Delta=\langle\oo_i,\oo_j',\beta_l:i,j\ge2\rangle$ correspond to the
underdetermined ODE system
 $$
y^{(m)}=(z^{(n)})^2,\ v_i'=g_i(x,y,z,z',\dots).
 $$
For general ODE systems $\E$ the proof is absolutely similar.
 \com{
Let $\oo_\z$ be the invariant co-frame, such that the first two forms satisfy the
structure equations $d\oo_1=0$, $d\oo_1'=0$, and the others belong to $\op{Ann}(\Delta)$
(cf. frame for the Monge equation from Section \ref{S31}).

Let us split the coordinates on $M$ into $(x,u_\z,z_n)$, where $x$ is one-dimensional, $u_\z$
multi-dimensional expressing all involved derivatives except $z_n$, which is the "vertical coordinate".
The distribution has basis $\D_x,\p_{z_n}$ and its annihilator are Cartan
forms $\oo_\z=du_\z-u_{\z+1}\,dx$.

The manifold $\tilde M$ has coordinates $(x,u_\z,v_i,z_n)$ and we add the forms
$\tilde\oo_i=dv_i-\b_i$ as in the previous proof.
In order to treat them as Cartan forms we need to choose potential $\b_i$ of $\a_i$
in the form $f_idx$ modulo $\langle\oo_\z:\z\ne1,1'\rangle$
(functions $f_i=f_i(x,u,z_n)$ are not invariant indeed). This is always possible,
as we can make the change $\b_i\mapsto\b_i-d[\int\b(\p_{z_n})dz_n]$ to avoid the
differential $dz_n$ in $\b_i$ and then substitute all $du_\z$ via $\oo_\z$.

Finally we get our integrable extension as the original equations plus the new ones: $v_i'=G_i(x,u,z_n)$.
Notice that solutions to the new system $\tilde\E$ can be obtained from the solutions
of $\E$ by quadratures, so that we can call this extension integral).
 }
 \end{proof}

 \com{
Notice that solutions of the integrable extension $\ti\E$ can be found through the solutions
of $\E$ by quadratures. The reason for this is that an Abelian group $V$ acts by symmetries on $\tilde\Delta$,
though S.\,Lie quadrature theorem does not apply directly\footnote{The condition is that solvable group acts on
an {\em integrable distribution\/} \cite{DL}, while our $\tilde\Delta$ is completely non-holonomic.},
but after certain reduction.
 }

Returning to rank 2 distributions we observe that any fundamental GNLA $\m$ with
gradings $(2,1,2,\dots)$ is a sequence of successive central extensions of the unique GNLA
with gradings $(2,1,2)$ and this implies Theorem \ref{Thm3}.

\subsection{Examples}\label{S53}

In this subsection we demonstrate how Proposition \ref{prop14} and
Theorem \ref{thm17} work in practice.

\smallskip

{\bf 5D$\to$6D.} Let us start with the unique GNLA $\mathfrak{f}_{1,2}$ with the grading $(2,1,2)$.
Its multiplication table in the basis $e_1,e_1',e_2,e_3,e_3'$ is
(we list all non-zero commutators):
 $$
[e_1,e_1']=e_2,\ [e_1,e_2]=e_3,\ [e_1',e_2]=e_3'.
 $$
This is equivalent to the structure equations in the dual co-basis:
 \begin{equation}\label{CH-StrEq}
d\oo_1=0,\ d\oo_1'=0,\ d\oo_2=\oo_1'\we\oo_1,\ d\oo_3=\oo_2\we\oo_1,\ d\oo_3'=\oo_2\we\oo_1'.
 \end{equation}
The group $H^2(\m)=H^2_4$ (of pure grading 4) is 3-dimensional,
the gauge group is $\mathfrak{H}=\op{GL}_2=\op{GL}(g_{-1})$ and the quotient $\mathbb{P}(H^2_4)/\mathfrak{H}$
consists of 3 points
represented by the following 3 classes (called parabolic $\mathfrak{p}_6$, hyperbolic $\mathfrak{h}_6$
and elliptic $\mathfrak{ell}_6$ 6D algebras, because $H^2(\m)$ bears the natural conformal
metric of Lorenzian signature (2,1), see \cite{AK}).

To calculate the integral extensions we need to realize the above abstract Maurer-Cartan forms.
This is easy since they are dual to the vector fields $e_1=-\D_x$, $e_1'=\p_{z_2}$, $e_2=\p_{z_1}+2z_2\p_y$,
$e_3=\p_z$, $e_3'=2\p_y$:
 \begin{multline}\label{xyz}
\oo_1=-dx,\ \oo_1'=dz_2,\ \oo_2=dz_1-z_2dx,\\
\oo_3=dz-z_1dx,\ \oo_3'=\tfrac12(dy-z_2^2dx)-z_2(dz_1-z_2dx).
 \end{multline}

\smallskip

$\mathfrak{p}_6$: The cohomology class of $\a=\oo_3\we\oo_1$ corresponds to the extension
 \begin{equation}\label{p6+}
d\oo_4=\oo_3\we\oo_1
 \end{equation}
or, in the dual (vector) form, to the additional commutator relation
 $$
[e_1,e_3]=e_4.
 $$
The flat model, calculated from (\ref{xyz}) as $\oo_3\we\oo_1=-d(z\,dx)$, is the
1-dimensional integral extension obtained from Hilbert-Cartan equation by $v'=z$,
whence the model Monge ODE is
 $$
\E_{1,3}:\ y'=(v''')^2.
 $$

\smallskip

$\mathfrak{h}_6=\mathfrak{f}_{2,2}$: The cohomology class of $\a=\oo_3\we\oo_1'+\oo_3'\we\oo_1$
corresponds to the extension $d\oo_4=\a$ or, in the dual form, to
 $$
[e_1,e_3']=e_4=[e_1',e_3].
 $$
The flat model, calculated from (\ref{xyz}) as $\a=d\b$ with
 $$\b=(z_1z_2-\frac12y)dx-z_2dx\equiv-\frac12y\,dx\!\mod\langle\oo_2,\oo_3,\oo_3'\rangle,$$
is the 1-dimensional integral extension obtained from Hilbert-Cartan equation by $v'=y$
(constant can be absorbed), whence the model Monge ODE
 $$
\E_{2,2}:\ v''=(z'')^2.
 $$

\smallskip

$\mathfrak{ell}_6$: The cohomology class of $\a=\oo_3\we\oo_1+\oo_3'\we\oo_1'$
corresponds to the extension $d\oo_4=\a$ or, in the dual form, to
 $$
[e_1,e_3]=e_4=[e_1',e_3'].
 $$
The flat model, calculated from (\ref{xyz}) as $\a=d\b$ with
 $$
\b=-(z+\frac16z_2^3)dx-\frac12z_2dy+\frac12z_2^2dz_1\equiv-(z+\frac16z_2^3)dx \!\mod\langle\oo_2,\oo_3,\oo_3'\rangle,
 $$
is the 1-dimensional integral extension obtained from the Hilbert-Cartan equation by $v'=-(z+\frac16z_2^3)$.
The resulting systems is not representable as a single Monge and (after absorption of constants) equals
 $$
y'=(z'')^2,\ v'=z-(z'')^3.
 $$

\smallskip

{\bf 6D$\to$7D.} There are 8 different fundamental GNLA of dimension 7. We will consider only half of
them -- those which are extensions of $\mathfrak{p}_6$. The structure equation of $\mathfrak{p}_6$ are
(\ref{CH-StrEq})+(\ref{p6+}).

The cohomology group consists of two pure gradings $H^2(\mathfrak{p}_6)=H^2_4\oplus H^2_5$ and both summands
are 2-dimensional. The gauge group $\mathfrak{H}$ is the (3-dimensional) Borel subgroup of $\op{GL}(g_{-1})$.
Both quotients $\mathbb{P}(H^2_4)/\mathfrak{H}$ and $\mathbb{P}(H^2_5)/\mathfrak{H}$
consist of two points and so there are 4 central extensions.

\smallskip

$(2,1,2,1,1)_1$. Consider at first central extension of maximal grading 5. The fixed point of the action
is given by the cohomology class $[\oo_4\we\oo_1]$. The corresponding central extension is the GNLA
$\mathfrak{p}_7$ given by the additional commutation relation
 $$
[e_1,e_4]=e_5
 $$
and the integral extension is
 $$
y'=(z^{iv})^2.
 $$

\smallskip

$(2,1,2,1,1)_2$. Another central 1-dimensional extension is obtained with the cohomology class
$[\oo_4\we\oo_1'-\oo_3\we\oo_2]$ ($\mathfrak{H}$-orbit through it is open).
The corresponding central extension is given by 
 $$
[e_1',e_4]=e_5=[e_3,e_2]
 $$
and the integral extension is given by the system
 $$
y'=(z''')^2,\ v'=(z'')^2.
 $$

\smallskip

$(2,1,2,2)_1$. Now consider the central extensions of grading 4. The fixed point of the $\mathfrak{H}$-action
is given here by the class $[\oo_3\we\oo_1'+\oo_3'\we\oo_1]$.
The corresponding central extension is the GNLA
$\mathfrak{f}_{2,3}$ given by 
 $$
[e_1,e_3']=e_4'=[e_1',e_3]
 $$
and the integral extension is given by the Monge equation
 $$
y''=(z''')^2.
 $$

\smallskip

$(2,1,2,2)_2$. Finally, another central extensions of grading 4 corresponds to an open orbit
of the $\mathfrak{H}$-action and is given by the class $[\oo_3'\we\oo_1']$.
The corresponding central extension is given by 
 $$
[e_1',e_3']=e_4'
 $$
and the integral extension is given by the system
 $$
y'=(z''')^2,\ v'=(z''')^3.
 $$

\subsection{Extensions of parabolic models $\frak{p}_n$.}
The GNLAs $\frak{p}_n=\frak{f}_{1,n-3}$ are of special interest as the most symmetric cases.
It is obvious that $\frak{p}_{n+1}$ is a central 1-dimensional extension of $\frak{p}_n$.
We would like to study other possible central extensions of $\frak{p}_n$ of pure grading.


Notice that any fundamental GNLA with $\dim g_{-1}=2$ can be obtained from $(2,1,2)$ by
successive $d$-dimensional central extensions of only maximal grading. In studying the
latter for the sequence of parabolic models we observe the following $\Z_2$-periodic pattern.

 \begin{theorem}
The Lie algebra $\frak{p}_{n+1}$ is a maximal grading central extension of $\frak{p}_n$. For odd $n$ such an extension is
unique, while for even $n=2l$ there are two
maximal grading extensions $\frak{p}_{n+1}$ and $\frak{p}_{n+1}'$.

The second is however less symmetric (meaning $\dim\widehat{\frak{p}_{n+1}'}<\dim\widehat{\frak{p}_{n+1}}$)
and has no maximal grading central extensions.
 \end{theorem}

 \begin{proof}
The Maurer-Cartan equations of $\mathfrak{p}_n$ are (the index denotes the grading):
 \begin{multline*}
d\oo_1=0,\ d\oo_1'=0,\ d\oo_2=\oo_1'\we\oo_1,\ d\oo_3=\oo_2\we\oo_1,\ d\oo_3'=\oo_2\we\oo_1',\\
d\oo_4=\oo_3\we\oo_1,\ d\oo_5=\oo_4\we\oo_1,\ \dots\ d\oo_{n-2}=\oo_{n-3}\we\oo_1.
 \end{multline*}
The 2-cocycles in the maximal grading are
 $$
Z^2_{n-1}=\langle\oo_{n-2}\we\oo_1,
\oo_{n-2}\we\oo_1'-\oo_{n-3}\we\oo_2+\oo_{n-4}\we\oo_3-\dots\pm\oo_l\we\oo_{l-1}\rangle
 $$
for even $n=2l$ and $Z^2_{n-1}=\langle\oo_{n-2}\we\oo_1\rangle$ for odd $n=2l+1$

There are no coboundaries in this grading and so $\mathbb{P}H^2_{n-1}=\R\mathbb{P}^\epsilon$, where $\epsilon=0$
for odd $n$ and $\epsilon=1$ for even $n$. The action of 3-dimensional gauge group $\mathfrak{H}$ has two orbits in the
latter case: fixed point $[\oo_{n-2}\we\oo_1]$ and the complement, so that we can take the the second form of
$Z^2_{n-1}$ for $n=2l$ as a representative.

Extension through the first class is possible for any $n$ and the result is $\mathfrak{p}_{n+1}$. Extension
via the second class yields the GNLA $\mathfrak{p}'_{2l+1}$ for $n=2l$ with the following additional commutation
relation:
 $$
[e_1',e_{n-2}]=e_{n-1},\ [e_2,e_{n-3}]=-e_{n-1},\ [e_3,e_{n-4}]=e_{n-1}\ \dots\ 
 $$
It's easy to check that there are no 2-cocycles $Z^2_{2l}(\mathfrak{p}'_{2l+1})=0$,
and so $H^2_{2l}(\mathfrak{p}'_{2l+1})=0$.
Thus no further maximal central extensions exist.

Finally to calculate the Tanaka prolongation of $\mathfrak{p}'_{2l+1}$ notice that both directions $\R e_1$ and
$\R e_1'$ are canonical (invariant under automorphisms) in $g_{-1}$ because they can be uniquely characterized by
the properties $\op{ad}_{e_1}|_{g_{2-2l}}=0$, $\op{ad}_{e_1'}|_{g_{-3}}=0$.
Therefore the space $g_0$ is 2-dimensional Abelian
(diagonal $2\times 2$ matrices). Calculations similar to that of the proof of Proposition \ref{L2} show that
$g_{-1}=0$ and so $\dim\widehat{\mathfrak{p}'_{2l+1}}=2l+3$ is smaller than $\dim\widehat{\mathfrak{p}_{2l+1}}=4l-1$.
 \end{proof}

The flat ODE for $\mathfrak{p}_n$ is given by the model equation $\E_{1,n-3}$. The flat ODE for $\mathfrak{p}'_{2l+1}$
is not realized as a Monge equation, but only as a system. To find it one follows the scheme of Section \ref{S53}
and therefore obtains the following (non-standard) integral extension of $\E_{1,2l-3}$:
 $$
y'=\bigl(z^{(2l-3)}\bigr)^2,\ v'=\bigl(z^{(l-1)}\bigr)^2.
 $$

\section{Applications to integration of PDEs}\label{S6}

We conclude with some examples which illustrate how our study of rank 2 distributions
for Monge equations can be applied to the integration of overdetermined and determined PDEs.

\subsection{From underdetermined ODEs to overdetermined PDEs}

The examples here generalize those of Cartan \cite{C1}.
Specifically, Cartan considered compatible systems consisting of pairs of
scalar 2nd order PDEs on the plane with one common characteristic (which we denote as $2E_2$).
The equation $\E\subset J^2(\R^2,\R)$ from the internal viewpoint is a
6-dimensional manifold $N_6$ equipped with a rank 3 distribution.
This distribution always has a Cauchy characteristic and can therefore be reduced
to a rank 2 distribution on a 5-dimensional manifold $N_5$, which is necessarily equivalent
to a single Monge equation (\ref{E1}).

Now we consider the case of mixed order equations $\E$ consisting of pairs of compatible 2nd and 3rd order
PDEs ($E_2+E_3$). Again it can be seen that such systems can be reduced to rank 2 distributions,
now on 6-dimensional manifolds, but we would like to emphasize that the reduction is slightly more
complicated than in Cartan's case.

Indeed the system $\E$ is represented internally as an 8-dimensional manifold $M_8$
equipped with a rank 3 distribution $\mathcal{Q}$.
This distribution has one Cauchy characteristic, while its derived
distribution $\mathcal{Q}_2$ has two. Thus the pair $(M_8,\mathcal{Q})$
has two reductions - one by Cauchy characteristic for $\mathcal{Q}$ (quotient)
and one by Cauchy characteristic for the derived distribution $\mathcal{Q}_2$ (de-prolongation).

These two reductions commute, so that we can do them in any order, and in Examples 1 and 2 we
perform both at once. But for the purposes of Example 3 we shall need first to de-prolong $M_8$ to $M_7$
(with distribution $\bar{\mathcal{Q}}$ of rank 3), and then to quotient $(M_7,\bar{\mathcal{Q}})$ to $(M_6,\Delta)$
(rank of $\Delta$ is 2). These steps are summarized in the following sequence of maps:
 $$
(M_8,\mathcal{Q})\longrightarrow(M_7,\bar{\mathcal{Q}})\longrightarrow(M_6,\Delta).
 $$

In the first two examples $(M_6,\Delta)$ turns out to be one of the 6-dimensional symmetric models
from Theorem \ref{Thm1} (this can be detected a priori via calculation of the symmetry algebra and
concluding that the distribution is flat in the Tanaka sense). Then we
use their integral curves to integrate the corresponding PDE systems.

\smallskip

{\bf Example 1: $\bold{E_2+E_3\rightarrowtail\mathcal{E}_{1,3}}$.}
Consider the PDE system, written parametrically as
 \begin{equation}\label{Exa1}
u_{xx}=\t^5x^2,\ \ u_{xy}=\t^4x^2,\ \ u_{xyy}=\tfrac54\t\,u_{yyy}+\tfrac45\t^3x.
 \end{equation}
It is compatible and has a common characteristic $\x=\D_x-\tfrac54\t\D_y$.

Internally the equation is $M_8=\R^8(x,y,u,u_x,u_y,\t,u_{yy},u_{yyy})$ and the derived distribution
$\mathcal{Q}_2$ has a two-dimensional space of Cauchy characteristics
 \begin{multline*}
\Pi=\bigl\langle\hat\xi=\p_x-\tfrac54\t\p_y+(u_x-\tfrac54\t\,u_y)\,\p_u-\tfrac14\t^5x^2\,\p_{u_x}\\
+\tfrac{\t}4(4\t^3x^2-5u_{yy})\,\p_{u_y}+\tfrac45\t^3x\,\p_{u_{yy}}-\tfrac\t{2x}\p_\t,\ \p_{u_{yyy}}\bigr\rangle.
 \end{multline*}
The distribution $\Pi$ is integrable and the quotient by its leaves
 $$
M_6=M_8/\Pi\simeq\R^6(t,v,w,w_1,w_2,w_3)
 $$
can be identified with the space of invariants
\begin{equation}
\begin{gathered}\label{M8-M6}
t=y+\tfrac52\t\,x,\
v=\tfrac{256}{625}u_x+\tfrac{128}{625}\t^5x^3,\\
w=-\tfrac72\t^5x^4+\tfrac{25}8\t^2x^2\,u_{yy}+\tfrac52\t\,x\,u_y-x\,u_x+u,\\
w_1=u_y+\tfrac52\t\,x\,u_{yy}-3\t^4x^3,\
w_2=u_{yy}-\tfrac85\t^3x^2,\
w_3=-\tfrac{16}{25}\t^2x.
\end{gathered}
\end{equation}

 \com{
The operator of invariant differentiation (with respect to shifts along the foliation)
 $$
D=\tfrac4{5\t}\p_x+\tfrac4{5\t}u_x\,\p_u+\tfrac45\t^4x^2\,\p_{u_x}+\tfrac45\t^3x^2\,\p_{u_y}-\tfrac2{5x}\,\p_{\t}.
 $$
It acts on the invariants as follows
 $$
D(L_6)=L_5, D(L_5)=L_4, D(L_4)=-16L_1^2/25, D(L_3)=L_1^4, D(L_2)=1, D(L_1)=0.
 $$
 $$
L_2=t, L_6=u, (\tfrac{4}{5})^4L_3=v, L_5=u_1, L_4=u_2, -(\tfrac45L_1)^2=u_3
 $$
 }

The quotient distribution on the manifold $M_6$ equals
 $$
\Delta=\langle\p_t+w_3^2\p_v+w_1\p_w+w_2\p_{w_1}+w_3\p_{w_2},\p_{w_3}\rangle
 $$
and this corresponds to the model Monge equation $\E_{1,3}$:
 $$
v'=(w''')^2.
 $$

The inverse image of the general solution to this equation with respect to the projection map $\pi:M_8\to M_6$
given by (\ref{M8-M6}) (this introduces only 1 more parameter: $\pi^{-1}$ decomposes into preimage and prolongation)
yields the general solution of (\ref{Exa1}), written parametrically as
 \begin{gather*}
x=-\Bigl(\frac{8s}{25}\Bigr)^2f'''(t),\ \ y=t-s\,f'''(t),\\
u=f(t)-\frac{s^3}{20}f'''(t)^4+\frac{s^2}2 f''(t)\,f'''(t)^2
-s\,f'(t)\,f'''(t)-\frac{s^2}4f'''(t)\,g(t),
 \end{gather*}
where $g'(t)=f'''(t)^2$.

\smallskip

{\bf Example 2: $\bold{E_2+E_3\rightarrowtail\mathcal{E}_{2,2}}$.}
Now let's consider the system
 \begin{equation}\label{Exa2}
2u_{xx}=y^2-\t^2,\ \ 2u_{yy}=\t^{-2},\ \ u_{xyy}-\t^2u_{yyy}=\t^{-1},
 \end{equation}
which is again compatible with common characteristic $\x=\D_x-\t^2\D_y$.

Internally, the equation is $M_8=\R^8(x,y,u,u_x,u_y,\t,u_{xy},u_{yyy})$ and the derived
distribution $\mathcal{Q}_2$ has a two-dimensional space of Cauchy characteristics
 \begin{multline*}
\Pi=\bigl\langle\hat\xi=
\p_x-\t^2\,\p_y+(u_x-u_y\t^2)\p_u+(\tfrac{y^2-\t^2}2-u_{xy}\t^2)\p_{u_x}\\
+(u_{xy}-\tfrac12)\p_{u_y}+(y-\t)\p_{u_{xy}}-\t^2\p_\t, \p_{u_{yyy}}\bigr\rangle.
 \end{multline*}
The distribution $\Pi$ is integrable and the quotient by its leaves
 $$
M_6=M_8/\Pi\simeq\R^6(t,v,v_1,w,w_1,w_2)
 $$
is identified with the space of invariants
 \begin{gather*}
t=y-\t,\
v=\ln\t-1+2u-2\t u_y+2u_{xy}-2xy-\t\,x^2y+2\t\,x+\\
+\tfrac{x^2(y^2+\t^2)}2-2x\,u_x+2\t\,x\,u_{xy},\
v_1=x^2y+\tfrac1{\t}-\t\,x^2-2x\,u_{xy}+2u_y,\\
w=\tfrac{xy^2}2-y-\t\,xy+\t+\tfrac{\t^2x}2+\t\,u_{xy}-u_x,\\
w_1=xy-\t\,x-u_{xy},\
w_2=-\tfrac1{\t}+x.
 \end{gather*}
In these coordinates the quotient distribution $\Delta$ on $M_6$
coincides with the rank 2 distribution of the model Monge equation $\E_{2,2}$:
 $$
v''=(w'')^2.
 $$
Again the general solution of this equation provides the general solution to system (\ref{Exa2}) as
 \begin{multline*}
x=f''(t)+\frac1s,\ \
y=t+s,\ \
u=\frac{(2s+t)t}4f''(t)^2+\frac{s}2g'(t)+\frac{g(t)}2-\\
-\frac{2f'(t)s^2+
2f(t)s-t^2}{2s}f''(t)-\frac{f(t)}s-\frac{t}{2s}+\frac{t^2}{4s^2}-\frac12\ln(s),
 \end{multline*}
where $g''(t)=f''(t)^2$.

\smallskip

We remark that the 2nd order PDE in (\ref{Exa1}) $x^2u_{xx}^4=u_{xy}^5$
is a new example of a Darboux integrable equation. On the other hand, the 2nd order
PDE in (\ref{Exa2}) $2u_{xx}+(2u_{yy})^{-1}=y^2$ is not Darboux integrable and
can therefore only be solved by compatibility technique.

\smallskip

{\bf Example 3: $\bold{E_2+E_3\rightarrowtail 2\,E_2}$.}
In this example we show that our theory of integrable extensions for Monge equations
can be applied to construct integrable extensions of overdetermined PDE systems.
Informally speaking, if $\E$ and $\E'$ are two PDE systems whose reduction by Cauchy characteristics
and de-prolongations are rank 2 distributions corresponding to equations $\bar\E$ and $\bar\E'$, then
an integrable extension $\varphi:\bar\E\to\bar\E'$ can be lifted to
an integrable extension $\Phi:\E\to\E'$.

To be more precise, in the notations adopted at the beginning of this section
we have the following commutative diagram
(we mark by primes the canonical distributions of the manifolds $N_i$)
 $$\begin{CD}
(M_7,\bar{\mathcal{Q}}) @>{\Phi}>> (N_6,\mathcal{Q}') \\ @V{\pi}VV @VV{\pi'}V \\
(M_6,\Delta) @>{\varphi}>> (N_5,\Delta')
 \end{CD}$$
where the horizontal arrows are the integrable extensions
and the vertical arrows are the quotients of the rank 3 distributions $\bar{\mathcal{Q}}$
and $\mathcal{Q}'$ by their Cauchy characteristics.

We illustrate this construction with the example of Cartan's involutive symmetric model \cite{C1}
 \begin{equation}\label{Exa3}
U_{XX}=\tfrac13\t^3,\ U_{XY}=\tfrac12\t^2,\ U_{YY}=\t.
 \end{equation}
This system reduces to $\E_{1,2}$: $V'=(W'')^2$ with $V=V(T),W=W(T)$ and the projection $\pi'$
along the Cauchy characteristic defined by
 \begin{multline*}
T=-U_{YY},\
V=-2U+2XU_X+2YU_Y-Y^2U_{YY}-XYU_{YY}^2-\tfrac13X^2U_{YY}^3,\\
W=U_X-U_Y U_{YY}+\tfrac12YU_{YY}^2+\tfrac16X U_{YY}^3,\ \\
W_1=U_Y-YU_{YY}-\tfrac12X U_{YY}^2,\
W_2=Y+X U_{YY}\qquad\qquad
 \end{multline*}

Basing on the fact that $\E_{1,3}$  is an integrable extension of $\E_{1,2}$
and on the reduction from Example 1 of this section,
we will show that (\ref{Exa1}) is an integrable extension of (\ref{Exa3}).

Recall that $(M_7,\bar{\mathcal{Q}})$ is obtained from $(M_8,\mathcal{Q})$ by de-prolongation via
a Cauchy characteristic for $\mathcal{Q}_2$ not belonging to $\mathcal{Q}$. Thus $\bar{\mathcal{Q}}$
still carries a Cauchy characteristic and we get
$(M_7,\bar{\mathcal{Q}})\simeq(\E_{1,3},\C^{1,3})\times(\R,\R)$ and similarly
$(N_6,\mathcal{Q}')\simeq(\E_{1,2},\C^{1,2})\times(\R,\R)$ for the Cartan example,
because $(M_6,\Delta)\simeq(\E_{1,3},\C^{1,3})$ and $(N_5,\Delta')\simeq(\E_{1,2},\C^{1,2})$.

The central extension $\varpi:\bold{\mathfrak{p}_6\to\mathfrak{p}_5}$ induces the integrable extension
$\phi=\varphi\times1:(\E_{1,3},\C)\times(\R,\R)\to(\E_{1,2},\C')\times(\R,\R)$, where
 $$
\varphi(t,v,w,w_1,w_2,w_3)=(t,v,w_1,w_2,w_3)=(T,V,W,W_1,W_2).
 $$
By this $\phi$ we construct the required map $\Phi$ in the diagram, and then we
compute its inverse $\Phi^{-1}$ given by the formulae
 \begin{gather*}
x=-\bigl(\tfrac45X\bigl)^2(Y+X\,U_{YY})^{-1}, \qquad y=-U_{YY}-\tfrac52X,\qquad \\
u_x=\bigl(\tfrac54\bigr)^4
\bigl(-\tfrac13X^2U_{YY}^3-\tfrac12X(2Y+X^2)\,U_{YY}^2-(X^2+Y)Y\,U_{YY}\\
\hskip5cm +2X\,U_X+2Y\,U_Y-2U-\tfrac12XY^2\bigr),\\
\hskip-0.2cm u_y=\tfrac16X\,U_{YY}^3+(\tfrac54X^2+\tfrac12Y)U_{YY}^2+(\tfrac{25}{16}X^3-U_Y+\tfrac52XY)U_{YY}\\
\hskip6.2cm +U_X-\tfrac52X\,U_Y+\tfrac{25}{16}X^2Y,
 \end{gather*}
where $(X,Y,U,U_X,U_Y,\dots)$ are jet-coordinates on the equation $N_6$ (\ref{Exa3}) and
$(x,y,u,u_x,u_y,\dots)$ are jet-coordinates on the equation $M_7$. Prolongation of the latter
yields the required integrable extension $M_8$.

Notice that $U$ is not expressed in the above formulae but is obtained via quadrature, as is always
the case with flat integrable extensions.
We leave it to the reader to check that the integrability condition $D_x(u_y)=D_y(u_x)$
holds by virtue of (\ref{Exa3}) and satisfies (\ref{Exa1}).

\subsection{Relation to the method of Darboux}

According to \cite{AFV} any Darboux integrable single scalar PDE on the plane $\E\subset J^2(\R^2,\R)$
is a group quotient of a pair of rank 2 distributions with common
freely acting symmetry group $G$, that is
 $$
(\E,\C)=(M_1,\Delta_1)\times(M_2,\Delta_2)/G.
 $$
Notice that in this case $\dim\E=7$, $\dim\C=4$, but $\dim M_i$ can be sufficiently large,
depending on the order of intermediate integrals of~$\E$.

Our theory of rank 2 distributions of maximal symmetry restricts the possible types of the factors.

 \begin{theorem}
Suppose that $(M_i,\Delta_i)$ have Carnot algebras $\mathfrak{f}_{m_i,n_i}$.
Then the pairs $(m_1,n_1)$, $(m_2,n_2)$ with normalization $0<m_i\le n_i$,
$n_2-n_1=k\ge0$, and $m_1\le m_2$ for $k=0$ satisfy the condition
 $$
m_1+m_2+k\le7+\delta,
 $$
where $\d=1$ if $m_1=m_2=1$ or $m_1=1$, $k>0$, and $\d=0$ else.
Thus there are no more than 53 different admissible triples $(m_1,m_2,k)$.
 \end{theorem}

 \begin{proof}
The equality
 $$
\dim M_1+\dim M_2-\dim G=\dim\E=7
 $$
implies $\dim G=n_1+m_1+n_2+m_2-3$. Since $G$ acts as the symmetry group for both $(M_i,\Delta_i)$
and since the symmetry group of a distribution is bounded from above by Tanaka algebra
$\mathfrak{t}_{m,n}=\widehat{\mathfrak{f}_{m,n}}$, i.e. $\dim G\le\dim\mathfrak{t}_{m_i,n_i}$
we get that $\min\{\dim \mathfrak{t}_{m_i,n_i}:i=1,2\}\ge n_1+m_1+n_2+m_2-3$.
From Section \ref{S34} we know that $\dim\mathfrak{t}_{m_i,n_i}=2n_i+5$ for $m_i=1$ and $2n_i+4$ for $m_i>1$.
The claim follows.
 \end{proof}

Recall that distributions $\Delta_i$ corresponding to nondegenerate Monge equations (\ref{E1})  
have Carnot algebras $\mathfrak{f}_{m_i,n_i}$.
The estimate of the theorem is rough and can be sharpened by a detailed analysis of the admissible
$(M_i,\Delta_i)$ which actually arise for various types of equations $\E$.

\smallskip

{\bf Example} (\cite{AF}){\bf.}
The Goursat equation $u_{xy}=a(x,y)\sqrt{u_xu_y}$ is Darboux integrable if and only if
both factors $(M_i,\Delta_i)$ can be realized as rank 2 distributions of the Monge equation
 $$
y'=F(x,z,z',\dots,z^{(n)})
 $$
and the group $G$ is either Abelian or 1-step solvable of dimension $2n+1$.

In particular, for the Darboux integrable equations
 $$
u_{xy}=\frac{2n}{x+y}\sqrt{u_xu_y}
 $$
both factors coincide with (the prolongation of) the model equations $\E_{1,n}$.
In this case, the Lie algebra of $G$ is equal to the nilradical $\mathfrak{r}$ of 
the Lie algebra $\mathfrak{t}_{1,n}$ given in Section \ref{Ss33}.


\end{document}